\documentclass[12pt,a4paper]{amsart}
\usepackage{amsfonts}
\usepackage{amsthm}
\usepackage{amsmath}
\usepackage{amscd}
\usepackage[utf8]{inputenc} 
\usepackage[T1]{fontenc} 
\usepackage{t1enc}
\usepackage[mathscr]{eucal}
\usepackage{indentfirst}
\usepackage{graphicx}
\usepackage{graphics}
\usepackage{pict2e}
\usepackage{epic}
\numberwithin{equation}{section}
\usepackage[margin=2.9cm]{geometry}
\usepackage{epstopdf} 
\usepackage{amssymb}
\usepackage{bbm}
\usepackage{pxfonts}
\usepackage{yfonts}
\usepackage{enumerate}

\usepackage{xypic} 

\input xy 		
\xyoption{all} 

\newtheorem{theorem}{ Theorem}[section] 

\theoremstyle{remark}

\newtheorem{definition}{Definition}[section]   
\newtheorem{proposition}[theorem]{Proposition}
\newtheorem{corollary}[theorem]{Corollary}
\newtheorem{lemma}[theorem]{ Lemma}

\newcommand{\set}[1]{\left\{ #1 \right\}}
\newcommand{\upla}[1]{\left( #1 \right)}
\newcommand{\innerProd}[2]{\left \langle #1, #2 \right \rangle}

\newcommand{\Eijk}{\epsilon_{ijk}}
\newcommand{\Hijk}{\xi_{ijk}}

\newcommand{\vacuum}{\mathbbm{1}}				
\newcommand{\id}{{\bf id}}						
\newcommand{\Hc}{{ \mathcal{H}}}				
\newcommand{\Hf}{{ \mathfrak{H}}}				
\newcommand{\Cc}{{ \mathfrak{C}_0}}				
\newcommand{\Complexn}{\mathbbm{C}}				
\newcommand{\imu}{ {\bf i}}						
\newcommand{\Gg}{\mathfrak{g}}					
\newcommand{\VT}{V_{\mathbbm{T}^3}}				
\newcommand{\KacMoody}{\widehat{\Gg}}			
\newcommand{\HeisR}{{\bf Heis}(\mathbbm{R})}		
\newcommand{\heisR}{{\bf heis}(\mathbbm{R})}		
\newcommand{\HeisZ}{{\bf Heis}(\mathbbm{Z})}		

\newcommand{\Res}{{\bf res}}						

\newcommand{\Field}[1]{{\bf Field} \left( #1\right)}		

\newcommand{\Lfield}[1]{{\bf LField} \left( #1\right)}		

\newcommand{\NormallyOP}[1]{{\bf \colon} #1 {\bf \colon}}

\newcommand{\Hom}[1]{{\bf Hom} \left( #1 \right)}

\newcommand{\End}[1]{{\bf End} \left( #1 \right)}

\newcommand{\Ind}[2]{{\bf Ind}_{#1}^{#2}} 		
\newcommand{\CoInd}[2]{{\bf CoInd}_{#1}^{#2}} 		

\newcommand{\bk}[1]{\left[ #1 \right]} 			

\newcommand{\LOGz}{\log (z)}					
\newcommand{\LOG}[1]{\log \left( #1 \right)}	

\newcommand{\Abs}[1]{ \left \vert #1 \right \vert}

\def\cH{\mathcal{H}}
\def\fg{\mathfrak{g}}
\def\cO{\mathcal{O}}
\def\cL{\mathcal{L}}

\usepackage{hyperref}

\hypersetup{pdfpagemode={UseOutlines},
	bookmarksopen=true,
	bookmarksopenlevel=0,
	hypertexnames=false,
	colorlinks=true,
	citecolor=magenta,
	linkcolor=red,
	urlcolor=mdtRed,
	pdfstartview={FitV},
	unicode,
	breaklinks=true,
}

\author[Morales B. R.]{Bely Rodriguez Morales}

\title{Logarithmic Modules for chiral differential operators of nilmanifolds}


\email{belyto.rodriguez@gmail.com}


\AtBeginDocument{
		\hypersetup{pdftitle=Logarithmic Modules for chiral differential operators of nilmanifolds} 
		\hypersetup{pdfauthor=Bely R Morales} 
}

\begin{document}

\begin{abstract} 
	We describe explicitly the vertex algebra of (twisted) chiral differential operators on certain nilmanifolds and construct their logarithmic modules. This is achieved by generalizing the construction of vertex operators in terms of exponentiated scalar fields to Jacobi theta functions naturally appearing in these nilmanifolds. This provides with a non-trivial example of logarithmic vertex algebra modules, a theory recently developed by Bakalov. 
\end{abstract}

\maketitle



\section{Introduction}

To any smooth manifold $M$ (and a choice of $\omega \in H^3(M, \mathbb{Z})$) satisfying some mild topological properties (the first Pontryagin class vanishes), Malikov, Schechtmann and Vaintrob \cite{malikov1999chiral} and independently Beilinson and Drinfeld \cite{beilinson2004chiral} attach a sheaf of vertex algebras $\mathcal{O}^{ch}_M$ called the sheaf of \emph{chiral differential operators}. In the simplest case when $\omega =0$, locally on a coordinate patch $U$ with coordinates $\left\{ x^i \right\}_{i=1,\ldots,\dim M}$, the sections $\cO^{ch}_M(U)$ form an $\dim M$-dimensional $\beta\gamma$-system, i.e., the vertex algebra generated by fields $\left\{\beta_i, \gamma^i\right\}_{i = 1,\ldots,\dim M}$ satisfying the OPE
\[ 
\beta_i(z) \cdot \gamma^j(w) \sim \frac{\delta_{i,j}}{z -w}, \qquad \beta_i(z) \cdot \beta_j(w) \sim \gamma^i(z) \cdot \gamma^j(w) \sim 0. 
\]
On intersections of coordinate patches, the fields $\gamma^i$ change as coordinates do  while the fields $\beta_i$ change as vector fields. 

This construction works in the algebraic, holomorphic, real-analytic or $C^\infty$-setting. In this work we will be mainly concerned with the $C^\infty$-setting. Little is known about the structure of the global sections $V_M=\Gamma(M, \cO^{ch}_M)$ of this sheaf. Only recently in the context of supermanifolds, Bailin Song proved that, in the holomorphic setting, the vertex algebra $V_M$ coincides with the simple small $N=4$ super-vertex algebra at central charge $c=6$ when $M = T^*[1]N$ is the shifted cotangent bundle to a $K3$ surface $N$ \cite{song2016vector}. In this work we will describe explicitly $V_M$ when $M$ is the Heisenberg $3$-dimensional nilmanifold. \\


The vertex algebra $V_M$ (or rather its super-extension) is expected to play a central role in Mirror-Symmetry. In particular, for $M$ and $N$ a mirror pair of Calabi-Yau manifolds, one expects a natural isomorphism $V_M \simeq V_N$ of vertex algebras. Their characters are known to be equal by work of Borisov and Libgober \cite{borisov2000elliptic}.

If $M$ is non-simply-connected, a subtle phenomenon arises as one needs to consider non-trivial \emph{windings}. Aldi and Heluani showed in \cite{aldi2012dilogarithms}, based on ideas of C. Hull \cite{hull2009double}, that when $(M, \omega)$ is the three torus $\mathbb{T}^3$ with its generator of $H^3(\mathbb{T}^3, \mathbb{Z}) \simeq \mathbb{Z}$, or if $M$ is its mirror dual: the Heisenberg $3$-dimensional nilmanifold $N$ with vanishing $\omega$, the vertex algebra $V_M$ can be naturally represented in a Hilbert space. 
This Hilbert space 
is associated to a $6$-dimensional nilmanifold $Y$ which fibers over both $\mathbb{T}^3$ and $N$. \\

%

For certain $M$, we can describe explicitly $V_M$ in terms of a larger manifold fibering over $M$. Suppose that the $\omega$-twisted Courant algebroid $TM \oplus T^*M$ of $M$ is parallellizable. That is, there exists a global frame of vector fields $\left\{ \beta_i \right\}$ and dual basis of differential forms $\left\{ \alpha^i \right\}$ such that $[\beta_i, \beta_j]_{Lie} + \iota_{\beta_i} \iota_{\beta_j} \omega$ is a constant combination of $\beta_i$'s and $\alpha^i$'s, and $\mathrm{Lie}_{\beta_i} \alpha^j$ is a constant linear combination of the $\alpha^i$'s. In other words, there exists a Lie algebra $\fg$, $\dim \fg = 2 \dim M$, with a symmetric invariant bilinear pairing of signature $(\dim M, \dim M)$, and a trivialization $T M \oplus T^*M \simeq \fg \times M$. The Courant-Doffman bracket of the frame $\left\{ \beta_i, \alpha^i \right\}$ is given by the bracket in $\fg$.

The approach, following ideas of C. Hull \cite{hull2009double} and exploited for example in \cite{bouwknegt2004t} is that one may try to find a manifold $N$ with the property that $\dim N = 2 \dim M$. It fibers over $M$, $N \twoheadrightarrow M$ and its parallelizable, such that $TN \simeq \fg \times N$, that is, the \emph{Lie bracket} of vectors in a frame is identified with the Lie bracket of $\fg$. In this situation we consider the $\fg$-module $C^\infty(N)$, the Kac-Moody affinization $\hat{\fg}$ of $\fg$ and its induced module from $C^\infty(N)$. We have an embedding $C^\infty(M) \hookrightarrow C^\infty(N)$ given by pullback, inducing the sequence of embeddings:
\[ 
V^1(\fg) \subset \Ind{\hat{\fg}_+}{\hat{\fg}} C^\infty(M) \subset \cH = \Ind{\hat{\fg}_+}{\hat{\fg}} C^\infty(N), 
\]
where the first module is induced from the constant function $1$, coincides with the vacuum module for the algebra $\hat{\fg}$ and is known to be a vertex algebra. The second module coincides with the vertex algebra $V_M$ and is here represented as a subspace of $\cH$. \\

Let $G$ be the unipotent Lie group with Lie algebra $\fg$. It is also an extension of $\mathbb{R}^3$ by $\mathbb{R}^3$. Let $\Gamma \subset G$ be the subgroup generated by a basis of the quotient $\mathbb{R}^3$ of $G$. It is a cocompact subgroup, the quotient $Y = G/\Gamma$ is a six-dimensional nilmanifold which is a non-trivial $\mathbb{T}^3$-fibration over $\mathbb{T}^3$. In fact we have the central extensions:
\begin{equation}\label{eq:central-extensions}
\begin{gathered}
0 \rightarrow \mathbb{R}^3 \rightarrow G \rightarrow \mathbb{R}^3 \rightarrow 0 \\
0 \rightarrow \mathbb{Z}^3 \rightarrow \Gamma \rightarrow \mathbb{Z}^3 \rightarrow 0 
\end{gathered}
\end{equation}
Showing $Y$ as a $\mathbb{T}^3 = \mathbb{R}^3/\mathbb{Z}^3$ fibration over $\mathbb{T}^3$. 

The Lie group $G$ acts on $L^2(Y)$ and its Lie algebra $\fg$ acts on $C^\infty(Y)$. We extend the $\fg = \fg \otimes t^0 \subset \hat{\fg} = \fg [t,t^{-1}] \oplus \mathbb{C}K$ action on $C^\infty(Y)$ into a representation of $\hat{\fg}_+ = \fg [t] \oplus \mathbb{C}K$ by letting $K$ act by $1$ and $a_n$ act by $0$ if $n \geq 0$. The vector space $\cH$ is the corresponding $\hat{\fg}$-induced module $\cH = \Ind{\hat{\fg}_+}{\hat{\fg}} C^{\infty}(Y)$.

Properly speaking $\cH$ is its $L^2$ completion, but we will not care about unitarity properties in this work. 

Notice that the constant function $1$ defines an embedding $V^1(\fg) \hookrightarrow \cH$ of the vacuum representation of $\hat{\fg}$ into $\cH$. As it is well known $V^1(\fg)$ is a vertex algebra and this embedding makes $\cH$ into a $V^1(\fg)$-module. 
Consider now the three Torus $\mathbb{T}^3 = \mathbb{R}^3 / \mathbb{Z}^3$, the morphism $Y \twoheadrightarrow \mathbb{T}^3$ provides an embedding $C^\infty(\mathbb{T}^3) \hookrightarrow C^\infty(Y)$. It is easy to see that this is an embedding of $\fg$-modules. The induced $\hat{\fg}$-module coincides with $V_{\mathbb{T}^3}$, that is 
\[ 
V^1(\fg) \subset V_{\mathbb{T}^3} \simeq \Ind{\hat{\fg}_+}{\hat{\fg}} C^\infty(\mathbb{T}^3) \subset \cH. 
\]
However a little work is required to check that $\cH$ is a vertex algebra module over $V_{\mathbb{T}^3}$. The fields associated to vectors $f \in C^\infty(\mathbb{T}^3)$ involve explicitly logarithms of the formal variable $z$. The situation is very similar to that of the lattice vertex algebra where the logarithms only appear exponentiated, hence appealing to the identity $\exp(\log(z)) = z$ one can get rid of them. 

The situation with the Heisenberg nilmanifold is a quite different. Any line $\cL \subset \mathbb{R}^3$ determines a central character $\chi_\cL : Z(G) \simeq \mathbb{R}^3 \rightarrow \mathbb{R}$ of $G$. We can view $\cL$ as a one dimensional subgroup of $G$ (in the quotient $\mathbb{R}^3$). The subgroup $K = \ker \chi_\cL \oplus \cL \subset G$ is normal and its cokernel 
\[ 
0 \rightarrow K \rightarrow G \rightarrow \HeisR \rightarrow 0,	
\]
is the $3$-dimensional real Heisenberg group. If the line $\cL$ is generated by an element of $\Gamma$, this sequence is compatible with $\Gamma$ in the sense that there exists an analogous sequence 
\[ 
0 \rightarrow K_\Gamma \rightarrow \Gamma \rightarrow \HeisZ \rightarrow 0, 
\]
whose quotient is now the integer Heisenberg group. This construction shows $Y$ as a fibration over the Heisenberg nilmanifold $N = \HeisR/ \HeisZ$ (it is not hard to see that the fiber is also a three torus $\mathbb{T}^3$). 

We obtain thus an embedding $C^\infty(N) \hookrightarrow C^\infty(Y)$. As before it is easy to see that this is an embedding of $\fg$-modules. It turns out that the induced $\hat{\fg}$-module is also isomorphic to the vertex algebra $V_N$:
\[ 
V^1(\fg) \subset V_N \simeq \Ind{\hat{\fg}_+}{\hat{\fg}} C^\infty(N) \subset \cH. 
\]
This time however, logarithms are unavoidable. In fact, the fields associated to vectors of $V_N$ have explicit logarithms of $z$ on them when acting on $\cH$. It is only by restricting to $V_N \subset \cH$ that they disappear by use of the same identity $\exp(\log(z))=z$. However, when analyzing the action of $V_N$ on $\cH$ these logarithms remain, making $\cH$ a \emph{logarithmic module} over $V_N$. \\

In order to describe explicitly the fields of $V_N$ and their action on $\cH$, we need to use certain results from harmonic analysis. In particular, since the representation of $G$ in $L^2(Y)$ is unitary, it decomposes into direct sum of irreducible representations. These representations turn out to be induced from unitary irreducible representations of the real Heisenberg group, and by the Stone-von Neumann theorem they are unique once we choose a central character. One can choose explicit cyclic vectors for these representations: they are given by appropriate constant (the vacuum vector), exponential (functions from $\mathbb{T}^3$), or Jacobi theta functions. We construct vertex operators associated to these Jacobi theta functions in complete analogy as how one constructs vertex operators associated to exponential functions. These operators however, carry an explicit dependency on the logarithm of the formal variable. We show by explicit computation the locality and translation invariant property as well as the axioms for a logarithmic module as in \cite{bakalov2016twisted}. 
The main results of this work  are:

\begin{theorem}
	$\Hc$ has the structure of  $\VT$-module.
\end{theorem}

\begin{theorem}
	$\Hc$ has the structure of logarithmic $V_N$-module.
\end{theorem}



\section{Quantum Fields and Vertex Algebras}

Let $V$ be a vector space over $\Complexn$, the quantum fields on $V$ are defined as $\Field{V} =\Hom{V, V(\!(z)\!)}$ where $V(\!(z)\!) = V[\![z]\!][z^{-1}]$ denotes the space of Laurent series on $V$; i.e. a field on $V$ is a formal series $a(z) = \sum_{n \in \mathbbm{Z}} a_{(n)} z^{-1-n}$  where $a_{(n)} \in \End{V}$ and for each $v \in V$,  $a_{(n)} v = 0$ for $n$ large enough. 

Two quantum fields $a(z_1), b(z_2)$ are called local if there is $N \in \mathbbm{N}$ such that
\begin{equation}	\label{LocalityVA}
(z_1-z_2)^N \bk{a(z_1), b(z_2)} = 0.
\end{equation}

\begin{lemma}[Dong]\cite{kac1998vertex}		\label{DongsLemmaVertexAlgebras}
	Let $a(z),b(z),c(z)$ be pairwise local fields on $V$ then $a(z)_{(n)}b(z), \partial_z a(z),$ $ b(z), c(z)$   $ n \in \mathbbm{Z}$ are also pairwise local fields. 
\end{lemma}

The $n$-product of two local fields is defined as 
\begin{equation}		\label{nProductBakalov}
\left( a(z_1)_{(n)} b(z_2) \right)(z) v  = \partial_{z_1}^{(N - 1 - n)} \left. \left( (z_1 - z_2)^N a(z_1)b(z_2)v \right)\right|_{z_1=z_2=z}
\end{equation}
for $v \in V, n < N$, and $\left( a(z_1)_{(n)} b(z_2) \right)(z) v = 0$ if $n \geq N$.\\

\noindent Given and operator $A$ we will use the notation $A^{(k)} = \frac{A^k}{k!}$.\\

Given a field $a(z) = \sum_{n \in \mathbbm{Z}} a_{(n)} z^{-1-n}$ the \emph{annihilation} and \emph{creation} parts of $a(z)$ are defined respectively as:
\[
a(z)_{-} = \sum_{n \geq  0} a_{(n)} z^{-1-n},
\]
\[
a(z)_{+} = \sum_{n \leq -1} a_{(n)} z^{-1-n};
\]
The normally ordered product of two fields $a(z_1), b(z_2)$ is defined by
\[
\NormallyOP{a(z_1)b(z_2)}  = a(z_1)_{+} b(z_2) + b(z_2) a(z_1)_{-}.
\]

\begin{definition}
	A vertex algebra is a vector space $V$, a distinguished vector  $\vacuum \in V$ and linear map
	\[
	Y_z: V \rightarrow \Field{V}, \quad v \mapsto Y(v,z),
	\]
	such that the following axioms are satisfied:\\
	\begin{tabular}{l l}
		(vacuum axiom)			& $Y(\vacuum, z) = \id$, \space	$Y(v,z) \vacuum \in V[\! [z] \!]$, \space $Y(v,z)\vacuum |_{z=0} = v$;\\
		(translation invariance) &		$\bk{T, Y(v, z)} = \partial_z Y(v, z)$;\\
		(locality axiom)		&	For every $v_1,v_2 \in V$ there is $N$ large enough such that
	\end{tabular}
	\[
	\upla{z_1 - z_2}^N \bk{Y_{z_1}(v_1), Y_{z_2}(v_2)} = 0.
	\]
	
	Where the translation endomorphism $T \in \End{V}$ is defined as $Tv = \partial_z Y(v,z) \vacuum |_{z=0}$.
\end{definition}

{\noindent \it Remark:} There are several equivalent approach to define vertex algebra \cite{de2006finite}.

Following \cite{carter2005lie} let $\Gg$ be a Lie algebra with a non degenerate symmetric invariant bilinear form $\innerProd{\cdot}{\cdot}: \Gg \times \Gg \rightarrow \Complexn$,  for instance, every finite dimensional semisimple Lie algebra has such bilinear form.  The Kac-Moody affine Lie algebra $\KacMoody$ is defined as vector space by $\KacMoody = \Gg[t,t^{-1}]\oplus \Complexn K$ with the commutator
\[
\bk{at^m, bt^n} = \bk{a,b}t^{m+n} + m \innerProd{a}{b} \delta_{m,-n} K,
\]
where $K$ is central. Let us introduce the notation $a_n = at^n$.

Consider the subalgebra of the Kac-Moody affine algebra given by $\mathfrak{g}[t] \oplus \Complexn K$ and its one dimensional representation $\Complexn \vacuum$, where $K$ acts by multiplication by a given scalar $k$ and the elements of $\mathfrak{g}[t]$ acts by zero.   

\begin{proposition}\cite{kac1998vertex}	\label{KacsMoodyVertexAlgebra}
	The $\KacMoody$ module 
	\[
	V^k(\Gg) = \Ind{\Gg[t]\oplus \Complexn K}{\KacMoody} \vacuum	 \simeq 	\mathcal{U}\left(\KacMoody \right) \otimes_{\mathcal{U}\left( \Gg[t] \oplus \Complexn K \right)} \Complexn \vacuum
	\]
	has a vertex algebra structure.
\end{proposition}

The map $Y_z: V^k(\Gg) \rightarrow \Field{V^k(\Gg)}$ is defined as
\[
Y(a_{-1} \vacuum, z) = \sum_{n \in \mathbbm{Z}} a_n z^{-1-n}, \quad a \in \Gg,
\]
and in general for the generators of $V^k(\Gg)$
\begin{equation}	\label{KacMoody_statefields}
Y(a_{-n_{k_1}}^1 \cdots a_{-n_{k_r}}^r \vacuum, z) = \NormallyOP{\partial_z^{(n_{k_1}-1)} Y(a_{-1}^1 \vacuum, z) \cdots \partial_z^{(n_{k_r}-1)} Y(a_{-1}^r \vacuum, z)} .
\end{equation}

This vertex algebra $V^k(\Gg)$ is known as \emph{the universal affine vertex algebra of level $k$} or as \emph{the Kac-Moody vertex algebra of level $k$}. 

\begin{definition}
	A  module over a vertex algebra $V$ is a vector space $W$ equipped with a linear map $Y_z: V \rightarrow \Field{W}$ such that:
	\begin{itemize}
		\item $Y(\vacuum)	=	\id$	
		\item $Y(a_{(n)}b)	=	Y(a)_{(n)} Y(b)$	for all $n \in \mathbbm{Z}$.
	\end{itemize}
\end{definition}



\section{Logarithmic Quantum Fields and Logarithmic Modules}

\subsection{Logarithmic Fields}
It is convenient to extend the notion of quantum fields defined before to include logarithms, i.e., it is often needed to have the notion of logarithm in the formal theory of fields, in this section the basic results of logarithmic quantum fields will be stated following the ideas developed by Bojko Bakalov in \cite{bakalov2016twisted}. Let us start by introducing the formal variable $\LOGz$ which intuitively can be thought as the logarithm of $z$. Since there are now two formal variables we have two possible derivations 
\[
D_z = \partial_z + z^{-1} \partial_{\LOGz},	\quad	\quad	D_{\LOGz} =	z\partial_z	+	\partial_{\LOGz}.
\]
Notice that here we are using the derivatives $D_z$ and $D_{\LOGz}$ instead of $\partial_z$ and $\partial_{\LOGz}$ because the former derivations carry formally the data coded in the analytic equation "$\partial_z \LOGz = \frac{1}{z}$" while the latter derivations do not.

Let $W$ be a vector space over $\Complexn$, let $\alpha \in \Complexn/ \mathbbm{Z}$ and define 
\[
{\bf LField}_{\alpha}(W)	=	\Hom{W, W[\LOGz] [\![z]\!]z^{-\alpha}},
\]
the space of logarithmic quantum fields on $W$ is defined to be
\[
\Lfield{W}	=	\bigoplus_{\alpha \in \Complexn / \mathbbm{Z}}	{\bf LField}_{\alpha}(W) .
\]

\noindent The logarithmic fields will be denoted as $a(z)$ instead of $a(\LOGz, z)$ when no confusion arise.

\begin{definition}
	Two logarithmic fields $a(z_1)$, $b(z_2)$ are local if for $N >> 0$ holds:
	\begin{equation}	\label{LocallityLogModules}
	(z_1 - z_2)^N \bk{a(z_1), b(z_2)} 	=	0.
	\end{equation}
\end{definition}

\begin{definition}
	The $n$-product of two local logarithmic fields $a(z_1)$ and $b(z_2)$ is defined as
	\begin{equation}	\label{nProductBakalovLogFields}
	\left(a(z_1)_{(n)} b(z_2) \right) (z) w	=	D_{z_1}^{(N-n-1)} \left. \left((z_1 - z_2)^N a(z_1)b(z_2)w \right)\right|_{z_1 = z_2 = z}
	\end{equation}
	for $w \in W$ and $n < N$. For $n \geq N$ the $n$-product is defined by $\left(a(z_1)_{(n)} b(z_2) \right) = 0$.
\end{definition}

\noindent {\it Remark:} Even when the expression $\LOGz$ is a formal variable we can define formally
\begin{eqnarray}
\LOG{xy}			&	=	&	\LOG{x}		+	\LOG{y},	\\
\LOG{\frac{x}{y}}	&	=	&	\LOG{x}		-	\LOG{y},	\\
\LOG{1-x}			&	=	&	-	\sum_{n > 0}	\frac{x^n}{n},
\end{eqnarray}
therefore the expression $\LOG{z_1-z_2}$ may be interpreted in the following way: 
\[
\LOG{z_1 - z_2}		=	\LOG{z_1}	+	\LOG{1	-	\frac{z_2}{z_1}}	=	\LOG{z_1}	-	\sum_{n > 0}	\frac{z_1^{-n} z_2^n}{n}.
\]

It is easy to derive the following properties from the Leibniz rule
\begin{eqnarray}
( D_z a )_{(n)} b							&	=	&	-n a_{(n-1)}b,	\label{DerivativeNProduct}	\\
D_z \left( a _{(n)} b\right)				&	=	&	( D_z a )_{(n)} b	+	a_{(n)} ( D_z b ),	\\
\left( \partial_{\LOGz} a_{(n)} b \right)	&	=	&	\left( \partial_{\LOGz} a\right)_{(n)} b	+	a_{(n)} \left( \partial_{\LOGz} b \right).
\end{eqnarray}

Once again there is a Dong's Lemma for logarithmic fields:

\begin{lemma}\label{DongsLemmaLogModules}
	Let $a(z),b(z),c(z)$ be pairwise local logarithmic fields then \\
	\begin{center}		
		\begin{tabular}{c l}	
			{\bf a.}	&	$a(z)_{(n)}b(z)$ and $c(z)$ are local fields for all $n \in \mathbbm{Z}$,	\\
			{\bf b.}	&	$D_z a(z)$, $b(z)$ and $D_{\LOGz} a(z)$	are pairwise local.	
		\end{tabular}
	\end{center}
\end{lemma}
\begin{proof}
	The part {\bf a} is proven in \cite{bakalov2016twisted}, for the part  {\bf b}  just notice that $D_z a(z) = a(z)_{(-2)} \id$,  $D_{\LOGz} = zD_z$ then use part {\bf a}.
\end{proof}

In order to define the normally ordered product for logarithmic fields some extra step is required, for $\alpha \in \Complexn / \mathbbm{Z}$ select a representative $\alpha_0$ such that $-1 < Re (\alpha_0) \leq 0$ then any element $a(z) \in {\bf LField}_{\alpha}(W)$ can be uniquely expressed as 
\[
a(z)	=	\sum_{n \in \mathbbm{Z}} a_n \left( \LOGz \right) z^{-n-\alpha_0},
\]
where for every $w \in W$ holds $a_n \left( \LOGz \right) w = 0$ for $n >> 0$. The annihilation and creation parts of $a(z)$ are defined:
\[
a(z)_{-}	=	\sum_{n \geq 1} a_n \left( \LOGz \right) z^{-n-\alpha_0},
\]
\[
a(z)_{+}	=	\sum_{n \leq 0} a_n \left( \LOGz \right) z^{-n-\alpha_0},
\]
and this concepts can be extended linearly to $\Lfield{W}$; then the normally ordered product of logarithmic fields is defined by the usual formula
\[
\NormallyOP{a(z_1)b(z_2)}  = a(z_1)_{+} b(z_2) + b(z_2) a(z_1)_{-}.
\]
The \emph{propagator} of two logarithmic fields $a(z_1)$, $b(z_2)$ is defined as
\[
P(a, b; z_1, z_2) = \bk{a(z_1)_{-}, b(z_2)} = a(z_1)b(z_2) - :a(z_1)b(z_2):.
\]
The propagator can be used to compute the $n$-products \cite{bakalov2016twisted}:
\begin{proposition}
	If $a(z_1)$, $b(z_2)$	are local logarithmic fields then the $n$-product for $n \geq 0$ can be computed by the formula
	\[
	\left(a(z_1)_{(n)} b(z_2) \right) (z) w	=	D_{z_1}^{(N-n-1)} \left. \left((z_1 - z_2)^N P(a, b; z_1, z_2)w \right)\right|_{z_1 = z_2 = z}
	\]
	where $N$ is large enough such that the equation \ref{LocallityLogModules} holds and $n < N$. 
\end{proposition}

\subsection{Logarithmic Modules} Logarithmic modules over vertex algebras are a generalization of the usual notion of modules over vertex algebras but allowing logarithmic quantum fields, formally we define:

\begin{definition}
	A logarithmic module over a vertex algebra $V$ is a vector space $W$ equipped with a linear map $Y_z: V \rightarrow \Lfield{W}$ such that:
	\begin{itemize}
		\item $Y(\vacuum)	=	\id$	
		\item $Y(a_{(n)}b)	=	Y(a)_{(n)} Y(b)$	for all $n \in \mathbbm{Z}$.
	\end{itemize}
\end{definition}

Moreover, if $V$ is a vertex algebra equipped with an automorphism $\varphi$ and $W$ is a logarithmic module over $V$ such that $Y(\varphi a) = e^{2\pi i D_{\LOGz}} Y(a)$ holds for every $a \in V$, then $W$ is called a \emph{$\varphi$-twisted logarithmic module} \cite{bakalov2016twisted}.

Let $W$ be a vector space and let $\mathcal{W} \subseteq \Lfield{W}$ be a collection of logarithmic fields which are pairwise local, denote by $\overline{\mathcal{W}}$ the smallest $\Complexn[D_{\LOGz}]$ submodule of $\Lfield{W}$ containing $\mathcal{W} \cup \set{\id}$ and closed under $n$-products; then, because of Proposition \ref{DongsLemmaLogModules}, $\overline{\mathcal{W}}$ is again a collection of pairwise local logarithmic fields.

\begin{theorem}[Bakalov] \label{BakalovTheorem}
	Let $W$ be a vector space and $\overline{\mathcal{W}}$ be defined as above. Then $\overline{\mathcal{W}}$ with the $n$-product of logarithmic fields has the structure of vertex algebra where the vacuum vector is $\id$ and the translation operator is $D_z$.
\end{theorem}

From this it becomes clear that $W$ is a logarithmic module over $\overline{\mathcal{W}}$, just take the map $Y: \overline{\mathcal{W}} \rightarrow \Lfield{W}$ to be the inclusion map; moreover, it is $e^{2\pi i D_{\LOGz}} $-twisted module. 	

\begin{corollary}
	Let $V$ be a vertex algebra and $W$ a vector space, then giving a logarithmic $V$-module structure on $W$ is equivalent to give a vertex algebra morphism $V \rightarrow \overline{\mathcal{W}}$ for a local collection of logarithmic fields $\mathcal{W} \subseteq \Lfield{W}$, moreover, if $V$ is equipped with an automorphism $\varphi$ then the module will be twisted if and only the associated vertex algebra morphism transforms $\varphi$ into $e^{2\pi i D_{\LOGz}} $.
\end{corollary}



\section{Functions on the Double Twisted Torus}
\label{section2.01}

\subsection{Double twisted torus}
Let $V$ be a 3-dimensional real vector space and consider $G$ the extension 
\[
\xymatrix{0 \ar[r] & \wedge^2 V \ar[r] & G \ar[r] & V \ar[r] & 0}
\]
with internal law
\[
\upla{v, \zeta}\upla{v', \zeta'}  = \upla{v+v', \zeta + \zeta'+ v\wedge v'}, \quad v,v'\in V, \zeta,\zeta'\in \wedge^2 V
\]
making $G$ into a group. Using a coordinate system $\set{x^i,x^*_i}, \quad i=1,2,3$, where $\set{x^i}$ are the coordinates on the canonical basis $\set{e^i}$ of $V$ and $\set{x^*_i}$ are coordinates on the basis $\set{e^*_i = \Eijk e^j \wedge e^k}$ of $\wedge^2 V$, this product translates as
\[
\upla{x^i,x^*_i}\upla{y^i,y^*_i} = \upla{x^i + y^i, x^*_i + y^*_i + \frac{1}{2} \Eijk x^jy^k},
\]
where $\Eijk$ denotes the totally antisymmetric tensor. The \emph{double twisted torus $Y$} is defined as the quotient of $G$ modulo the subgroup $\Gamma$ generated by $e^i,  i=1,2,3$ the standard basis of $V\simeq \mathbbm{R}^3$. The tangent bundle TY is trivialized by the left invariant vector fields of $G$: 
\[
\alpha^i = \partial_{x^*_i}, \quad \beta_i = \partial_{x^i} - \frac{1}{2} \Eijk x^j \partial_{x^*_k} , 
\]
being $\bk{\beta_i, \beta_j} = \Eijk \alpha^k$ the only non trivial commutators, therefore they span a Lie algebra $\Gg$; moreover this Lie algebra is equipped with a non degenerate symmetric invariant bilinear form 
\[
\innerProd{\beta_i}{\alpha^j} = \delta_{i,j}.
\]

Now consider the space of polynomials $\Complexn[x^i, x^*_i] $ which is a $\Gg$-module via the restriction of the action on $C^\infty(G)$, let $\KacMoody = \Gg[t,t^{-1}]\oplus \Complexn K$ be the affine Kac-Moody Lie algebra associated to $\Gg$ and extend the action for the elements $a_n = at^n$, $a \in \Gg$, $n \geq 1$ by zero and make $K$ act as the identity, then define the $\KacMoody$-module
\begin{equation}\label{Hdefinition}
\Hf = \Ind{\Gg[t] \oplus \Complexn K}{ \KacMoody} \Complexn[x_i, x^*_i].
\end{equation}

Notice that $\Hf$ has naturally the structure of $V^1(\Gg)$-module.\\

Let's start by defining some operators on $\Hf$ that will be useful later, particularly when we  try to fit $\Hf$ into an algebraic structure:\\

Define the operators $x^i_n := - \frac{1}{n} \alpha^i_n$ for $n \neq 0$, note that those operators commute with each other, define $x^i_0$ acting on an element of $\Complexn[x^i,x^*_i]$ as $f \mapsto x^i f$, impose that $\bk{x^j_n, x^i_0} = 0$ and $\bk{\beta_{j,n}, x^i_0} = \delta_{i,j} \delta_{n,0} K$; therefore $x^i_0$ can be extended to $\Hf$.\\

Define the operators $W^i := \alpha^i_0$ on $\Hf$ for $i=1,2,3$ and 
\[
P_i := \beta_{i,0} + \Eijk x^j_0 W^k - \frac{1}{2} \Eijk \sum_m m x^j_{-m} x^k_m,
\]
note that since all the $\alpha^i_n$ commute with each other then all the $x^i_n$ and $W^j$ commute with each other. Also define the operators 
\[
x^*_{i,n} := - \frac{1}{n} \left( \beta_{i,n} + \Eijk x^j_n W^k - \frac{1}{2} \Eijk \sum_m m x^j_{n-m} x^k_m \right),	\quad n \neq 0,
\]
the operators $x^*_{i,0}$ will be defined acting on functions as $f \mapsto x^*_i f$ with the commutation relations $\bk{x^*_{i,0}, x^j_n} = \bk{x^*_{i,0}, x^*_{j,0}} = 0$, $\bk{x^*_{i,0}, \beta_{j,n}} = \frac{1}{2} \Eijk x^k_n$, $\bk{W^j, x^*_{i,0}} = \delta_{i,j} K$.	\\

\noindent {\bf Remark:} The operators $P_i$ and $x^*_{i,n}$ are well defined because even when the sum appearing in the last term runs over the integers it is actually finite since $x^i_m$ acts by zero for $m$ big enough.
\\

It would be convenient to compute explicitly for later reuse all the commutators of the previously defined operators. It is obvious that

\begin{eqnarray}	\label{TrivialCommutationRelation}
\bk{\alpha^j_m, x^i_n}			& = &	0,	\\
\bk{\beta_{j,m}, x^i_n}			& = &	\delta_{i,j} \delta_{n,-m} K,	\\
\bk{\alpha^i_m, W^j}			& = &	0,	\\
\bk{\beta_{j,m}, W^i}			& =	&	0,	\\
\bk{\alpha^j_{m}, x^*_{i,n}}	& = &	\delta_{i,j} \delta_{n,-m} K	\\
\bk{\alpha^j_m, P_i}			& = &	0	\\
\bk{P_i, W^j}					& = &	0.	
\end{eqnarray}

For $\beta^j_m$ and $x^*_{i,n}$ with $n \neq 0$ it holds

\begin{eqnarray*}
	\bk{\beta_{j,m}, x^*_{i,n}}	& = &	-\frac{1}{n} \bk{\beta_{j,m}, \beta_{i,n}}	-	\frac{\epsilon_{ipq}}{n} \bk{\beta_{j,m}, x^p_n W^q}	+	\frac{\epsilon_{ipq}}{2n} \sum_s s \bk{\beta_{j,m}, x^p_{n-s} x^q_s}	\\
	& = &	\frac{\Eijk}{n} \alpha^k_{n+m}	-	\frac{\epsilon_{ipq}}{n} \delta_{j,p} \delta_{m,-n} W^q	+	\frac{\epsilon_{ipq}}{2n} \sum_s s \bk{\beta_{j,m}, x^p_{n-s}} x^q_s	+	\frac{\epsilon_{ipq}}{2n} \sum_s s x^q_s \bk{\beta_{j,m}, x^p_{n-s}}	\\
	& = &	\frac{\Eijk}{n} \alpha^k_{n+m}	-	\frac{\epsilon_{ijk}}{n} \delta_{m,-n} W^k	+	\frac{\epsilon_{ipq}}{2n} \sum_s s \delta_{j,p} \delta_{m,s-n} x^q_s	+	\frac{\epsilon_{ipq}}{2n} \sum_s s \delta_{j,q} \delta_{m,-s} x^p_{n-s} 	\\
	& = &	\frac{\Eijk}{n} \alpha^k_{n+m}	-	\frac{\Eijk}{n}  \delta_{m,-n} W^k	+	\frac{\Eijk}{2n} (m+n) x^k_{m+n}	+	\frac{\Eijk}{2n} m x^k_{m+n} \\
	& = &	\scalebox{0.68}{$(1-\delta_{m,-n})\frac{\Eijk}{n} \alpha^k_{n+m}	+	\delta_{m,-n}\frac{\Eijk}{n} \alpha^k_{n+m}	-	\frac{\Eijk}{n}  \delta_{m,-n} W^k	+	(1-\delta_{m,-n})\frac{\Eijk}{2n} (m+n) x^k_{m+n}	+	(1-\delta_{m,-n})\frac{\Eijk}{2n} m x^k_{m+n}	+	\delta_{m,-n}\frac{\Eijk}{2n} m x^k_{m+n}$}	\\
	& = &	\scalebox{0.78}{$(1-\delta_{m,-n}) \left( \frac{\Eijk}{n} \alpha^k_{n+m}	+	\frac{\Eijk}{2n} (m+n) x^k_{m+n}	+	\frac{\Eijk}{2n} m x^k_{m+n}	 \right)	+	\delta_{m,-n} \left(\frac{\Eijk}{n} \alpha^k_{n+m}	-	\frac{\Eijk}{n}  W^k	+	\frac{\Eijk}{2n} m x^k_{m+n}	 \right)$}	\\
	& = &	\scalebox{0.78}{$(1-\delta_{m,-n}) \left( -\frac{\Eijk}{n} (m+n) x^k_{n+m}	+	\frac{\Eijk}{2n} (m+n) x^k_{m+n}	+	\frac{\Eijk}{2n} m x^k_{m+n}	 \right)	+	\delta_{m,-n} \left(\frac{\Eijk}{n} \alpha^k_{0}	-	\frac{\Eijk}{n}  W^k	-	\frac{\Eijk}{2} x^k_{m+n}	 \right)$}	\\
	& = &	-(1-\delta_{m,-n}) \frac{\Eijk}{2}  x^k_{m+n}		-	\delta_{m,-n} \frac{\Eijk}{2} x^k_{m+n}	\\
	& = &	-	\frac{\Eijk}{2} x^k_{m+n},
\end{eqnarray*}

notice that $x^*_{i,0}$ was defined in a way such that the previous formula is also satisfied. 

For $\beta^j_m$ and $P^i$ the bracket is computed as follows

\begin{eqnarray*}
	\bk{\beta_{j,m}, P^i}	& = &	\bk{\beta_{j,m}, \beta_{i0}}	+	\epsilon_{ipq} \bk{\beta_{j,m}, x^p_0 W^q}	-	\frac{\epsilon_{ipq}}{2} \sum_{n} n \bk{\beta_{j,m}, x^p_{-n} x^q_n}	\\
	& = &	-\Eijk \alpha^k_m	+	\epsilon_{ipq} \bk{\beta_{j,m}, x^p_0}W^q	-	\frac{\epsilon_{ipq}}{2} \sum_{n} n \bk{\beta_{j,m}, x^p_{-n}} x^q_n	-	\frac{\epsilon_{ipq}}{2} \sum_{n} n x^p_{-n}\bk{\beta_{j,m},  x^q_n}	\\
	& = &	-\Eijk \alpha^k_m	+	\epsilon_{ipq} \delta_{j,p}\delta_{m,0} W^q	-	\frac{\epsilon_{ipq}}{2} \sum_{n} n \delta_{j,p} \delta_{m,n} x^q_n	-	\frac{\epsilon_{ipq}}{2} \sum_{n} n  \delta_{j,q} \delta_{m,-n} x^p_{-n}	\\
	& = &	-\Eijk \alpha^k_m	+	\epsilon_{ijk} \delta_{m,0} W^k	-	\frac{\epsilon_{ijk} m }{2}  x^k_m	+	\frac{\epsilon_{ikj} m}{2}   \delta_{m,-n} x^k_m	\\
	& = &	-\Eijk \alpha^k_m	+	\epsilon_{ijk} \delta_{m,0} W^k	-	\epsilon_{ijk} m   x^k_m	\\
	& = &	\Eijk(1 - \delta_{m,0}) ( -\alpha^k_m	-	m   x^k_m )	+	\epsilon_{ipk} z\delta_{m,0} (-	\alpha_0^k	+	W^k)	\\
	& = &	\Eijk(1 - \delta_{m,0}) ( m x^k_m	-	m   x^k_m )	+	\epsilon_{ipk} z\delta_{m,0} (-	W^k	+	W^k)	\\
	& = &	0.
\end{eqnarray*}

Similarly, the remaining commutators can be computed obtaining:

\begin{eqnarray}	\label{NoTrivialCommutationRelation}
\bk{\beta_{j,m}, x^*_{i,n}}			& = &	-	\frac{\Eijk}{2} x^k_{m+n},	\\
\bk{\beta_{j,m}, P_i}				& = &	0,								\\
\bk{P_i, P_j}						& = &	-	\Eijk W^k,						\\
\bk{x^*_{i,n}, x^*_{j,0}}			& =	&	\frac{\Eijk}{2n} x^k_n,		\\
\bk{x_{i,m}, x^*_{j,n}}				& = &	\frac{m+n}{2mn} \Eijk + \frac{1}{m^2} \Eijk W^k \delta_{m,-n}.	
\end{eqnarray}


Let us define the fields
\[
\alpha^i(z)	=	\sum_{n} \alpha^i_n z^{-1-n},
\]
\[
\beta_i(z)	=	\sum_{n} \beta_{i,n} z^{-1-n},
\]
and the logarithmic fields
\[
x^i (z)	=	W^i\LOGz + \sum_{n \in \mathbbm{Z}} x^i_n z^{-n},
\]
\[
x^*_i (z)	=	P_i\LOGz + \sum_{n \in \mathbbm{Z}} x^*_{in} z^{-n} + \frac{\LOGz}{2} \Eijk W^j x^k(z),
\]

Note that following differential equation holds
\begin{equation}	\label{DiffEquationAlpha}
D_z	x^i(z)	=	\alpha^i(z),
\end{equation}

it will be useful to find a similar equation for the derivative of $x^*_i(z)$. The following equation can be obtained by straightforward computation:

\begin{equation}	\label{DiffEquationBeta}
D_z	x^*_i(z)	=	\beta_i(z)	+	\frac{1}{2} \Eijk x^j(z) D_z x^k(z).
\end{equation}



\section{Modules and the Double Twisted Torus}

\subsection{Fibration over the 3 Torus}

As it was noticed before, the group $G$ acts on $L^{2}(G/ \Gamma)$ as left translations and therefore the Lie algebra $\Gg$ acts on smooth functions on $G / \Gamma$ as left invariant vector fields, i.e., $\Gg$ acts on a dense subspace of $L^2(G/ \Gamma)$; so similarly to the case of polynomials it is possible to obtain a $\KacMoody$-module out of it inducing
\[
\Hc = \Ind{\Gg[t] \oplus \Complexn K}{\KacMoody} C^{\infty}(G/ \Gamma), 
\]
where as usual $\alpha^i_n$ and $\beta_{i,n}$ act as zero when $n>0$ and $K$ acts as the identity. 

Every function $f \in C^{\infty}(G/ \Gamma)$ can be interpreted as a rapidly decreasing smooth function in six variables $f = f(x^i, x^*_i)$ invariant under the action of $\Gamma$ on the right, i.e. for every $(\gamma^i, \gamma^*_i) \in \Gamma$ holds $f(x^i, x^*_i)	=	f(x^i + \gamma^i, x^*_i + \gamma^*_i + \frac{1}{2} \Eijk x^j \gamma^k)$. Such functions can be decomposed in a\emph{ Fourier series} with respect to the orthonormal system $\set{e^{2\pi  \imu \omega_i x^*_i}}_{\omega \in \mathbbm{Z}^3}$ as:
\[
f(x^i, x^*_i)	=	\sum_{\omega \in \mathbbm{Z}^3} e^{2\pi  \imu \omega_i x^*_i} f_{\omega}(x^i),
\]
where $f_{\omega}$ satisfies $f_{\omega}(x^i + \gamma^i)	=	e^{-\pi \imu \Eijk \omega_i x^j \gamma^k} f_{\omega}(x^i)$.

Define 
\[
C_{\omega} = \set{ e^{2\pi  \imu \omega_i x^*_i} f;  \quad f: \mathbbm{R}^3 \rightarrow \Complexn,  \quad  f(x^i + \gamma^i)	=	e^{-\pi \imu \Eijk \omega_i x^j \gamma^k} f(x^i)},
\]
then
\[
L^{2}(G/ \Gamma)	\simeq	\bigoplus_{\omega \in \mathbbm{Z}^3} C_\omega,
\]
specifically for $\omega = 0$ we have
\[
C_0	=	\set{f: f(x^i + \gamma^i)	=	f(x^i)} = \bigoplus_{\rho \in \mathbbm{Z}^3} \Complexn e^{2\pi \imu \rho_i x^i}.
\]
Define also 
\[
\VT	=	\Ind{\Gg[t] \oplus \Complexn K}{\KacMoody} C_0	=	\Ind{\Gg[t] \oplus \Complexn K}{\KacMoody} \bigoplus_{\rho \in \mathbbm{Z}^3} \Complexn e^{2\pi \imu \rho_i x^i}	=	\bigoplus_{\rho \in \mathbbm{Z}^3} \Ind{\Gg[t] \oplus \Complexn K}{\KacMoody}  \Complexn e^{2\pi \imu \rho_i x^i}.
\]

Geometrically we have a $\mathbbm{T}^3$ fibration over the $\mathbbm{T}^3$
\[
\xymatrix{
	\mathbbm{T}^3 \ar[r]		&	G/ \Gamma \ar[r]	&	\mathbbm{T}^3,
} 
\]
and embeddings
\[
\Complexn	\subset		C^{\infty}(\mathbbm{T}^3)	\subset		C^{\infty}(G / \Gamma),
\]
which leads to 
\[
V^1(\Gg)  \vacuum	\subset \VT = \Ind{\Gg[t] \oplus \Complexn K}{\KacMoody} C^{\infty}(\mathbbm{T}^3)	\subset		\Hc = \Ind{\Gg[t] \oplus \Complexn K}{\KacMoody} C^{\infty}(G / \Gamma).
\]
We know $V^1(\Gg)$ is a vertex algebra, i.e., it is the Kac-Moody vertex algebra of level $1$, it turns out $\VT$ is a vertex algebra too.

\begin{theorem}		\label{Hf0IsVertexAlgebra}
	$\VT$ is a vertex algebra.
\end{theorem}

The vertex algebra $\VT$ is called the \emph{vertex algebra of chiral differential operators} associated to $\mathbbm{T}^3$ whose existence is proven abstractly in \cite{malikov1999chiral}, however it would be convenient to give a prove of the theorem computing explicitly the quantum fields.

\begin{proof}
	Define the vacuum vector $\vacuum$ as the constant $1$ function, and consider the state field correspondence map
	\[
	Y: \VT \rightarrow \Field{\VT}
	\]
	\[
	e^{2\pi \imu \rho_i x^i} \mapsto \NormallyOP{ e^{2\pi \imu \rho_i x^i(z)} }	=	e^{2 \pi \imu \rho_i x^i_0} z^{2 \pi \imu \rho_i W^i} \exp \left( 2 \pi \imu \rho_i \sum_{n<0} x^i_n z^{-n}  \right) \exp \left( 2 \pi \imu \rho_i \sum_{n>0} x^i_n z^{-n} \right),
	\]
	\[
	\alpha^i_{-1} \vacuum	\mapsto	\alpha^i(z), 
	\]
	\[
	\beta_{i,-1} \vacuum	\mapsto	\beta_i(z).
	\]
	Let us quickly check the vertex algebra axioms :
	\[
	\left. Y(e^{2\pi \imu \rho_i x^i}, z) \vacuum \right|_{z=0}		=	\left. e^{2 \pi \imu \rho_i x^i_0} z^{ 2 \pi \imu \rho_i W^i} \exp \left( 2 \pi \imu \rho_i \sum_{n<0} x^i_n z^{-n}  \right) \exp \left( 2 \pi \imu \rho_i \sum_{n>0} x^i_n z^{-n} \right) \vacuum \right|_{z=0} = e^{2\pi \imu \rho_i x^i}.
	\]
	The fields of the form $Y(e^{2\pi \imu \rho_i x^i}, z)$ commute with each other because all the $x^i_n$ and $W^i$ commute, therefore they are local. The fields of the form $\alpha^i(z)$, $\beta_i(z)$ are pairwise local. Now because $\bk{\alpha^j_m, x^i(z_2)} = 0$ it is deduced that $\alpha^j(z_1)$ and $Y(e^{2\pi \imu \rho_i x^i}, z_2)$ commute. 
	
	The locality for the fields $\beta^j(z_1)$ and $e^{2\pi \imu \rho_i x^i(z_2)}$ is checked as follows:
	\[
	\bk{\beta_{j,n}, x^i(z_2)} =	\delta_{i,j}K z_2^n,
	\]
	then 
	\[
	\bk{\beta_{j,n}, e^{2\pi \imu \rho_i x^i(z_2)}} =	\delta_{i,j} 2\pi \imu \rho_i e^{2\pi \imu \rho_i x^i(z_2)} z_2^n,
	\]
	from this follows
	\[
	\bk{\beta_j(z_1), e^{2\pi \imu \rho_i x^i(z_2)}}	=	\sum_{n} \bk{\beta_{j,n}, e^{2\pi \imu \rho_i x^i(z_2)}} z_1^{-1-n}	 
	\]\[
	= \delta_{i,j} \sum_{n} 2\pi \imu \rho_i e^{2\pi \imu \rho_i x^i(z_2)} z_1^{-1-n} z_2^n	=	 \delta_{i,j} 2\pi \imu \rho_i e^{2\pi \imu \rho_i x^i(z_2)} \delta(z_1 , z_2), 
	\]
	therefore the fields $\beta^j(z)$ and $e^{2\pi \imu \rho_i x^i(z)}$ are local. The locality for any other pair of fields follows from Dong's Lemma \ref{DongsLemmaVertexAlgebras}.

	The last condition remaining to be proved is the translation invariance of the fields, let us define the translation endomorphism $T$ in $\VT$. Initially it is convenient to define $T$ acting on $x^i$, $T(x^i)$ should be a vector such that $Y( T(x^i), z) = \partial_z Y(x^i,z)$, but this equation is satisfied by $\alpha^i(z)$ because of the equation (\ref{DiffEquationAlpha}), so it becomes natural to define $T(x^i) = \alpha^i_{-1} \vacuum$.

	Now it easy to define $T$ on any function as
	\begin{equation}   \label{TonFunctionsExponentials}
	T \left( e^{2 \pi \imu \rho_i x^i} \right) 	=	2 \pi \imu \rho_i e^{2 \pi \imu \rho_i x^i} T(x^i)	=	2 \pi \imu \rho_i e^{2 \pi \imu \rho_i x^i} \alpha^i_{-1} \vacuum. 
	\end{equation}
	
	We define $T(\vacuum) = 1$ and in the same way as it is done in the Kac-Moody algebra we extend $T$ recursively by the formula $\bk{T, a_n} = -n a_{n-1}$ and impose the commutation relation $\bk{T, x^i_0} = \alpha_{-1}$, finally $T$ extends to the whole $\VT$ as a derivation of the normally ordered product. Note that $T$ was defined in a way so it satisfies translation invariance for the fields $\alpha^i(z)$ and $\beta_{i}(z)$, and for $x^i(z)$ holds
	\[
	\bk{T, x^i(z)}	=	\bk{T, W^i}\LOGz	+	\sum_n \bk{T, x^i_n} z^{-n}	=	\bk{T, \alpha^i_0}\LOGz	+	\sum_n - \frac{1}{n} \bk{T, \alpha^i_n} z^{-n}	
	\]\[
	=	\sum_n - \frac{1}{n} (-n) \alpha^i_{n-1} z^{-n}	=	 \sum_n \alpha^i_{n} z^{-1-n}	=	\alpha^i(z)	=	\partial_z x^i(z),
	\] 
	so now we can compute
	\[
	\bk{T, e^{2 \pi \imu \rho_i x^i(z)}}	=	2 \pi \imu \rho_i e^{2 \pi \imu \rho_i x^i(z)} \bk{T, x^i(z)}	=	2 \pi \imu \rho_i  e^{2 \pi \imu \rho_i x^i(z)}	\partial_z x^i(z)	=	\partial_z e^{2 \pi \imu \rho_i x^i(z)}.
	\]
\end{proof}

\begin{theorem}	\label{HF0Module}
	The space $\Hc$ has the structure of  $\VT$-module.
\end{theorem}

\begin{proof}
	We must define a  field for each vector of $\VT$. Set
	\begin{eqnarray*}
		Y( e^{2\pi \imu \rho_i x^i}, z)	&	=	&	e^{2\pi \imu \rho_i x^i(z)},	\\
		Y( \alpha^i_{-1}\vacuum, z)					&	=	&	\alpha^i(z),	\\ 
		Y( \beta_{i,-1}\vacuum, z)					&	=	&	\beta_i(z),
	\end{eqnarray*}
	and for the rest of the elements define the associated field by the normally ordered product, exactly as in the Kac-Moody algebra, for example
	\[
	Y(\alpha^i_{-1} e^{2\pi \imu \rho_j x^j})	= \NormallyOP{ \alpha^i(z) e^{2\pi \imu \rho_j x^j(z)}}.
	\]
	
	Note that we are defining the fields exactly as in the proof of theorem \ref{Hf0IsVertexAlgebra} and proving the locality condition for those fields we never used the fact that the logarithmic terms disappeared, i.e., what we actually prove there was that those logarithmic fields are pairwise local.
	
	Now in this way because of \ref{DerivativeNProduct} and because the normally ordered product of two fields is the $-1$-product,  the function $Y: \VT \rightarrow \Field{\Hc}$ commutes with the $n$-products when $n < 0$.
	For any two pairs of fields of the form $\alpha^i(z)$ and $\beta^j(z)$, it is trivial to see that the $n$-product condition holds.
	So it is only left to check it for pair of fields $\upla{\alpha^i(z), e^{2\pi \imu \rho_j x^j(z)}}$,  $\upla{\beta_i(z), e^{2\pi \imu \rho_j x^j(z)}}$, $\upla{ e^{2\pi \imu \rho_i x^i(z)}, e^{2\pi \imu \rho_i x^i(z)}}$ and for $n \geq 0$.
	
	As $e^{2\pi \imu \rho_i x^i(z)}$ and $e^{2\pi \imu \rho_j x^j(z)}$ commute we know that $ {e^{2\pi \imu \rho_i x^i(z)}}_{(n)} e^{2\pi \imu \rho_j x^j(z)} = 0$ for $n \geq 0$ but ${e^{2\pi \imu \rho_i x^i}}_{(n)} e^{2\pi \imu \rho_j x^j} = 0$ as well, the same situation repeats for the fields $\upla{\alpha^i(z), e^{2\pi \imu \rho_j x^j(z)}}$.
	
	For the last pair of fields $\upla{\beta_i(z), e^{2\pi \imu \rho_j x^j(z)}}$ we only need to compute for $n = 0$,
	\begin{eqnarray*}
		\beta_i(z)_{(0)} e^{2\pi \imu \rho_j x^j(z)}	&	=	&	\left. (z_1 - z_2) \bk{\beta_i(z_1)_{-}, e^{2\pi \imu \rho_j x^j(z_2)}} \right\vert_{z_1=z_2 = z}	\\
		&	=	&	2\pi \imu \delta_{i,j}\rho_j e^{2\pi \imu \rho_j x^j(z)}	\\
		&	=	&	Y\left( 2\pi \imu \rho_j \delta_{i,j} e^{2\pi \imu \rho_j x^j}, z\right)	\\
		&	=	&	Y\left( \beta_{i,-1}\vacuum_{(0)} e^{2\pi \imu \rho_j x^j}, z\right).
	\end{eqnarray*}
	
\end{proof}



\subsection{Fibration over the Heisenberg Nilmanifold} 

As it was explained before there is a $\mathbbm{T}^3$ fibration over the Heisenberg nilmanifold
\[
\xymatrix{
	\mathbbm{T}^3 \ar[r]		&	G/ \Gamma \ar[r]	&	N = \HeisR \left/ \HeisZ \right. ,
}
\]

the goal of this subsection is proving that $\Hc$ is a logarithmic module over the vertex algebra of chiral differential operators over the Heisenberg nilmanifold ($V_N$), we will do this by carefully restricting to a  subspace of $C^{\infty}(G/ \Gamma)$ such the vector fields identify with the Heisenberg Lie algebra, i.e., we will describe explicitly the structure of $V_N$ and compute the quantum fields. In order to achieve that it will be required to use some techniques from harmonic analysis on the Heisenberg group to deduce the structure of the space and therefore define the quantum vectors (logarithmic) fields for the algebra and for the module $\Hc$. The reader interested in a deeper study of  harmonic analysis of the Heisenberg group may consult  \cite{auslander2006abelian}. \\


Define the symbol $\xi_{ijk}$ as $\xi_{123} = \xi_{231} = \xi_{312} = 1$ and $\xi_{ijk} = 0$ for the remaining cases. Notice that it holds that $\xi_{ijk} - \xi_{ikj} = \Eijk$. 

It is convenient to change the coordinates
\begin{eqnarray*}
	x^i		&	\mapsto	&	x^i, \\
	x^*_i	& \mapsto	&	x^*_i + \frac{1}{2} \Hijk x^jx^k,
\end{eqnarray*}
so that the group law turns into
\[
(x^i, x^*_i) (y^i, y^*_i)	=	(x^i + y^i, x^*_i + y^*_i + \Hijk x^jy^k ),
\]
and the action of $\alpha^i$ and $\beta_j$  in these coordinates looks like
\begin{eqnarray*}	
	\alpha^i	&	=	&	\partial_{x^*_i}	\\
	\beta_i		&	=	&	\partial_{x^i} + \Hijk x^k \partial_{x^*_j}.
\end{eqnarray*}

We still have the Fourier type decomposition $L^{2}(G/ \Gamma)	\simeq	\bigoplus_{\omega \in \mathbbm{Z}^3} C_\omega$ but this time $C_{\omega}$ is given by
\[
C_{\omega} = \set{ e^{2\pi  \imu \omega_i x^*_i} f; \quad f: \mathbbm{R}^3 \rightarrow \Complexn,  \quad  f(x^i + \gamma^i)	=	e^{-2\pi \imu \omega_i \Hijk x^j \gamma^k} f(x^i)}.
\]
Consider vectors $\omega \in \mathbbm{Z}$ of the form $\omega = (0,0,n)$ 
\[
C_{(0,0,n)} = \set{ e^{2\pi  \imu n x^*_3} f_{0,0,n}: \quad f_{0,0,n}(x^i + \gamma^i) = e^{-2 \pi \imu n x^1 \gamma^2} f_{0,0,n}(x^i)},
\]
now  those functions $f_{0,0,n}$ can be decomposed into Fourier series once again
\[
f_{0,0,n} (x^1, x^2, x^3) = \sum_{m \in \mathbbm{Z}} e^{2 \pi \imu m x^3} f_{0,0,n,m}(x^1, x^2),
\]
from this it follows the decomposition of $C_{(0,0,n)} = \bigoplus_{m \in \mathbbm{Z}} C_{(0,0,n,m)}$. Let us take $m=0$ and consider the elements from $C_{(0,0,n,0)}$ for all $n \in \mathbbm{Z}$, i.e. the functions $f_n: \mathbbm{R}^2 \rightarrow \Complexn$ such that $f_n(x^1 + \gamma^1, x^2 + \gamma^2)	=	e^{-2 \pi \imu n x^1 \gamma^2} f_n(x^1, x^2)$. Define
\[
\Cc	=	\bigoplus_{n \in \mathbbm{Z}}  C_{(0,0,n,0)}	=	\set{\sum_n e^{2 \pi \imu n x^*_3} f_n(x^1, x^2)} \subset L^2(G / \Gamma).
\]

Note that restricting to elements of $G$ of the form $(x^1, x^2, 0, 0, 0, x^*_3)$ is the same as working in $\HeisR$, the (polarized) Heisenberg group 
\[
\set{
	\begin{pmatrix}
	1	&	x^1	&	x^*_3	\\
	0	&	1	&	x^2		\\
	0	&	0	&	1
	\end{pmatrix}
},
\]
so $\HeisR$ acts on $C_{0,0,n,0}$ by right translations, i.e., looking at the functions on $Y$ depending only on the variables $x^1$, $x^2$, $x^*_3$ is the same as looking at the Heisenberg nilmanifold $N = \HeisR / \HeisZ$.


\noindent {\it Remark:}	The space $\Cc$ constructed above using Fourier analysis can be described intrinsically as follows:  since the Heisenberg group is a central extension
\[
1	\rightarrow \mathbbm{R}	\rightarrow		\HeisR	\rightarrow		\mathbbm{R}^2	\rightarrow		1,
\]
and if we take a one dimensional complex representation of the center, i.e., a central character $\chi_n : \mathbbm{R} \rightarrow \Complexn^*$, $\chi_n(x^*_3) = e^{2\pi \imu n x^*_3}$ then $C_{(0,0,n,0)}	=	\CoInd{\mathbbm{R}}{\HeisR}	\Complexn_{\chi_n} $ and $\Cc$ can be expressed as
\[
\Cc	=	\bigoplus_{n \in \mathbbm{Z}}  C_{(0,0,n,0)}		=	\bigoplus_{n \in \mathbbm{Z}}  \CoInd{\mathbbm{R}}{\HeisR}	\Complexn_{\chi_n}.
\]
\\

Define for $n \neq 0$ and $m \in \set{0,1, \ldots \vert n \vert - 1}$ the k-linear map $\Theta_m: L^2(\mathbbm{R}) \rightarrow C_{(0,0,n,0)}$ as
\begin{equation}	\label{MapTheta}
\Theta_m(g) (x^1, x^2, 0, 0, 0, x^*_3)	=	e^{2 \pi \imu n x^*_3} \sum_{k \in \mathbbm{Z}} e^{2 \pi \imu (n k + m) x^1}  g(x^2 + k),	
\end{equation}

and for $\gamma^1, \gamma^2 \in \mathbbm{Z}$ it holds
\[
\Theta_m(g) (x^1+\gamma^1, x^2 + \gamma^2,0,0,0,x^*_3)	=	e^{2 \pi \imu n x^*_3} \sum_{k \in \mathbbm{Z}} e^{2 \pi \imu (n k + m)(x^1 + \gamma^1)} g(x^2 + \gamma^2 + k)	
\] 
\[
=	e^{2 \pi \imu n x^*_3} \sum_{k \in \mathbbm{Z}} e^{2 \pi \imu (n k + m) x^1} g(x^2 + \gamma^2 + k),	
\] 
making $s = k + \gamma^2$
\[
=	e^{-2 \pi \imu n x^1 \gamma^2} e^{2 \pi \imu n x^*_3} \sum_{s \in \mathbbm{Z}} e^{2 \pi \imu (n s + m) x^1} g(x^2 + s)	= e^{-2 \pi \imu n x^1 \gamma^2} \Theta_k(g) (x^1, x^2, 0, 0, 0, x^*_3),
\] 
then $\Theta_m(g) \in C_{0,0,n,0}$, so $\Theta_m$ is a well defined linear, because of the orthogonality relations between the exponentials the maps $\Theta_m$ are monomorphisms so then they define a unique monomorphism 
\[
\Theta: L^2(\mathbbm{R}) \otimes \Complexn^{\vert n \vert} \rightarrow C_{(0,0,n,0)}.
\]
Let $e^{2 \pi \imu n x^*_3}f$ be an element in $C_{(0,0,n,0)}$ then because $f(x^1 + 1, x^2)	=	f(x^1, x^2)$ it can be decomposed into Fourier series
\[
e^{2 \pi \imu n x^*_3} f(x^1, x^2)	=	e^{2 \pi \imu n x^*_3} \sum_{k \in \mathbbm{Z}} e^{2 \pi \imu k x^1} f_k(x^2)	=	\sum_{m=0}^{\Abs{n}-1} e^{2 \pi \imu n x^*_3} \sum_{k \in \mathbbm{Z}}  e^{2 \pi \imu (nk + m) x^1} f_{kn+m}(x^2), 
\] 
from the property $f(x^1, x^2 + \gamma^2) = e^{-2 \pi \imu n x^1 \gamma^2 }f(x^1, x^2)$ it follows $f_{nk+m}(x^2) = f_m(x^2 + k)$ so 
\[
e^{2 \pi \imu n x^*_3} f(x^1, x^2)	=	\sum_{m=0}^{\Abs{n}-1} e^{2 \pi \imu n x^*_3} \sum_{k \in \mathbbm{Z}}  e^{2 \pi \imu (nk + m) x^1} f_{m}(x^2 + k), 
\]
this means that any function in $C_{0,0,n,0}$ is uniquely determined by $f_0, f_1, \ldots, f_{\Abs{n}-1}$, i.e., there is a linear injection

\begin{eqnarray*}	
	\Phi: C_{(0,0,n,0)}	&	\rightarrow		&	L^2(\mathbbm{R}) \otimes \Complexn^{\Abs{n}}	\\
	f					&	\mapsto			&	\upla{f_0, \ldots, f_{\Abs{n}-1}},
\end{eqnarray*}
moreover $\Phi$ and $\Theta_m$ are inverse functions, so we can make $L^2(\mathbbm{R}) \otimes \Complexn^{\Abs{n}}$ a $\HeisR$-module.

Now because of the Stone-von Neumann theorem\cite{folland2016course}, $L^2(\mathbbm{R})$ is the only irreducible unitary representation of $\HeisR$ with the central character $\chi_n(t) = e^{2 \pi \imu n t}$  and there is a unique decomposition $C_{(0,0,n,0)} \simeq L^2(\mathbbm{R}) \otimes \Complexn^{p_n}$ being $\Theta$ an isomorphism and $\Abs{n} = p_n$. A fully detailed proof of this can be found in \cite{auslander2006abelian}.

\begin{proposition} \label{DescomposiciondeC00n0}
	If $n \neq 0$ then $C_{(0,0,n,0)} \simeq L^2(\mathbbm{R}) \otimes \Complexn^{\Abs{n}}$ and for $n=0$ it holds that $C_{(0,0,n,0)} \simeq L^2(\mathbbm{T}^2)$, i.e.,
	\[
	\Cc \simeq L^2(\mathbbm{T}^2) \oplus \bigoplus_{n \neq 0} L^2(\mathbbm{R}) \otimes \Complexn^{\Abs{n}} 
	\simeq \bigoplus_{\rho \in \mathbbm{Z}^2} \Complexn e^{2\pi \imu \rho_i x^i} \oplus \bigoplus_{n \neq 0} L^2(\mathbbm{R}) \otimes \Complexn^{\Abs{n}}. 
	\]
\end{proposition}

Proposition \ref{DescomposiciondeC00n0}  means that the $G$-module $\Cc$ identifies with the $\Gg$-module 
\[
\bigoplus_{\rho \in \mathbbm{Z}^2} \Complexn e^{2\pi \imu \rho_i x^i} \oplus \bigoplus_{n \neq 0} S(\mathbbm{R}) \otimes \Complexn^{\Abs{n}},
\]
where $S(\mathbbm{R})$ denotes the Schwartz space of rapidly decreasing smooth functions in $\mathbbm{R}$. We will also denote this space by $\Cc$. Since one space is the completion of the other one and it will always be clear to distinguish which one we are using.


Define the action of elements of $at^n \in t\Gg[t]$ on $\Cc$ by zero and the action of $K$ as the identity so it is possible now to induce
\[
V_N	=	\Ind{\Gg[t]\oplus \Complexn K}{\KacMoody} \Cc	=	\Ind{\Gg[t]\oplus \Complexn K}{\KacMoody} \left( \bigoplus_{\rho \in \mathbbm{Z}^2} \Complexn e^{2\pi \imu \rho_i x^i} \oplus \bigoplus_{n \neq 0} S(\mathbbm{R}) \otimes \Complexn^{\Abs{n}} \right)
\]
\[
=	\bigoplus_{\rho \in \mathbbm{Z}^2} \Ind{\Gg[t]\oplus \Complexn K}{\KacMoody} \Complexn e^{2\pi \imu \rho_i x^i} \oplus \bigoplus_{n \neq 0} \Ind{\Gg[t]\oplus \Complexn K}{\KacMoody} S(\mathbbm{R}) \otimes \Complexn^{\Abs{n}}
\]

\begin{theorem}			\label{Hf1IsVertexAlgebra}
	The space $V_N$ has a vertex algebra structure.
\end{theorem}

The vertex algebra $V_N$ is the \emph{vertex algebra of chiral differential operators} over the Heisenberg nilmanifold, once again instead of following \cite{malikov1999chiral} we will compute explicitly the quantum fields so it become easier to describe its logarithmic module.

{\noindent \it Remark.} Notice that after the change of coordinates previously done the fields $x^i(z)$ remain invariant but the fields $x^*_i(z)$ don not, specifically $x^*_3(z)$ transforms into
\[
x^*_3 (z) 	= P_3\LOGz + \sum_{i \in \mathbbm{Z}} x^*_{3,i} z^{-i}	+	\LOGz W^1\sum_{i \in \mathbbm{Z}} x^2_i z^{-i}	+ 	\frac{1}{2}\sum_{i,j \in \mathbbm{Z}} x^1_i x^2_j z^{-i-j}	+	\frac{1}{2} W^1 W^2 \left(\LOGz\right)^2,
\]
Moreover $P_3$ and $W^i$ act trivially on $V_N$ so all the terms with logarithms in $x^*_3(z)$ and $x^i(z)$ are zero.

\begin{proof}
	The vacuum vector $\vacuum \in V_N$ will be the $1$ constant function, let us start defining fields for the basis elements: once again the fields associated to elements of the form $a_{-n_{k_1}} \cdots a_{-n_{k_r}} \otimes 1$ with $a_{-n_{k_1}} \in \KacMoody$ will be the {\it same} as for the Kac-Moody vertex algebra, for elements $e^{2\pi \imu \rho_i x^i}$ we define
	\[
	Y(e^{2\pi \imu \rho_i x^i})	=	 :e^{ \left(2 \pi \imu \rho_i x^i(z) \right)}:	=	\exp \left(2 \pi \imu \rho_i x^i(z)_{+} \right) \exp \left(2 \pi \imu \rho_i x^i(z)_{-} \right),		
	\]
	which can also be written as
	\[		
	Y(e^{2\pi \imu \rho_i x^i})	=	e^{2 \pi \imu \rho_i x^i_0}  \exp \left(2 \pi \imu \rho_i \sum_{n<0} x^i_n z^{-n}  \right) \exp \left(2 \pi \imu \rho_i \sum_{n>0} x^i_n z^{-n} \right).
	\]
	
	Denote $F_m \in \Cc$ the image of $e^{- (x^2)^2}$ by $\Theta_m$, i.e. 
	\[
	F_m = \Theta_m(e^{- (x^2)^2})	=	\sum_{k \in \mathbbm{Z}} e^{2 \pi \imu (n k + m) x^1} e^{2 \pi \imu n x^*_3} e^{- (x^2 + k)^2 }, 	
	\]
	and define
	\[
	Y(F_m, z)	=	\sum_{k \in \mathbbm{Z}} :\exp \left( 2 \pi \imu n x^*_3(z) \right) \exp\left( 2 \pi \imu (n k + m) x^1(z) \right)  \exp \left(- (x^2(z))^2 - 2kx^2(z) - k^2 \right):,
	\]	
	since $\bk{x^*_3(z)_{\pm}, x^i(z)_{\mp}} = \bk{x^*_3(z)_{\pm}, \left(x^2(z)\right)_{\mp}} = \bk{x^i(z)_{\pm} , \left(x^2(z)\right)_{\mp}} = 0$ for $i=1,2$ one can write
	\[
	Y(F_m, z)	=	\sum_{k \in \mathbbm{Z}} \exp \left( 2 \pi \imu n x^*_3(z)_+ \right) \exp\left( 2 \pi \imu (n k + m) x^1(z)_+ \right)  \exp \left(- (x^2(z))^2_{+} - 2kx^2(z)_{+} - k^2 \right) \cdot
	\]\[
	\exp \left( 2 \pi \imu n x^*_3(z)_{-} \right) \exp\left( 2 \pi \imu (n k + m) x^1(z)_{-} \right)  \exp \left(- (x^2(z))^2_{-} - 2kx^2(z)_{-} \right),
	\]
	so for every $g \in \Cc$ it holds $Y(F_m) g \in V_N(\!(z)\!)$. Now for $a = \beta_1$, $a = \beta_2$ or $a = \alpha^3$ we would like to define $Y(f)$ for every  $f \in \bigoplus_{n \neq 0} S(\mathbbm{R}) \otimes \Complexn^{\Abs{n}}$ imposing  the following equation
	
	\begin{equation}	\label{EquationForDefiningFieldsofFunctions}
	\bk{a(z_1), Y(f, z_2)}	=	Y(a f, z_2) \delta(z_1, z_2).
	\end{equation}
	Define $Y(a F_m)$ by the formula
	\[
	Y(a F_m)	=	a_0 Y(F_m)	-	Y(F_m) a_0,
	\]
	clearly if $g \in \Cc$ then $Y(a F_m) g	\in V_N(\!(z)\!)$, note that since $L^2(\mathbbm{R})$ is an irreducible $\heisR$-mod then every element $f \in \bigoplus_{n \neq 0} S(\mathbbm{R}) \otimes \Complexn^{\Abs{n}}$ is obtained acting successively on $F_m$ with $\set{\beta_1, \beta_2, \alpha^3}$ and adding, so with the above formula it is proven that $Y(f)g \in V_N(\!(z)\!)$ for every functions $f,g$. It remains to prove that $Y(f) v \in V_N(\!(z)\!)$ for every function $f$ and every $v =  a_{n_k} \ldots a_{n_1} g \in V_N$, this can be done by induction on $k$, the base case ($k=0$) is already proven, for the recursive case just write $v = a_q w$ where $Y(f) w \in V_N(\!(z)\!)$, so expanding the equation \ref{EquationForDefiningFieldsofFunctions} we get
	\[
	Y(f, z_2) a(z_1) w	=	a(z_1) Y(f, z_2)w	-	Y(a_0f, z_2) \delta(z_1, z_2)w,
	\]
	multiplying by $z_1^q$ and taking residues we get
	\begin{eqnarray*}
		\Res_{z_1} z_1^q Y(f, z_2) a(z_1) w	&	=	&	\Res_{z_1} z_1^q a(z_1) Y(f, z_2) w	-	\Res_{z_1} z_1^q Y(a_0f, z_2) \delta(z_1, z_2) w,\\
		Y(f, z_2) a_q w	&	=	&	a_q Y(f, z_2) w	-	z_2^q Y(a_0f, z_2) \delta(z_1, z_2) w,		 	 \\
		Y(f, z_2) x		&	=	&	a_q Y(f, z_2) w	-	z_2^q Y(a_0f, z_2) \delta(z_1, z_2) w.
	\end{eqnarray*}
	but $a_q Y(f, z_2) w \in V_N(\!(z)\!)$	and	$z_2^q Y(a_0f, z_2) \delta(z_1, z_2) w \in V_N(\!(z)\!)$ because the induction hypothesis and then follows the desired result.
	
	From this $Y: V_N \rightarrow \Field{V_N}$ is fully determined since the field for the remaining  vectors is  determined by taking the normally ordered product of the above fields and by linearity.\\
	
	From the previous analysis it is also deduced that $Y(F_m) \vacuum \in V_N[\![z]\!]$ and from this it follows that $Y(v,x) \vacuum  \in V_N[\![z]\!]$ for all $v \in V_N$, and it is clear that $\left. Y(F_m,z) \vacuum \right|_{z=0} = F_m$ so it holds for every element in $V_N$.
	
	The computations to check the locality condition for the fields $\alpha^i(z), \beta_i(z), Y(e^{2\pi \imu \rho_i x^i}, z)$ are analogous to the ones done for the Kac-Moody vertex algebra and  the Theorem \ref{Hf0IsVertexAlgebra}. Since the field $x^3(z)$ commute with itself,  the fields $Y(F_{m_1}, z_1)$ and $Y(F_{m_2}, z_2)$ commute, therefore the only remaining pairs to check are $\set{\alpha^i(z_1), Y(F_m, z_2)}$ and $\set{\beta^i(z_1), Y(F_m, z_2)}$.
	
	From $\bk{\alpha^i_r, x^*_3(z_2)} = \delta_{i,3} z_2^r$ it follows that
	\[
	\bk{\alpha^i_r, e^{2 \pi \imu n x^*_3(z_2)}}	=	 2 \pi \imu n z_2^r e^{2 \pi \imu n x^*_3(z_2)} \delta_{i,3},
	\]
	and so
	\begin{eqnarray*}
		\bk{\alpha^i_r, Y(F_m, z_2)}	&	=	&	\scalebox{0.85}{$	\bk{\alpha^i_r, \sum_{k \in \mathbbm{Z}} \exp \left( 2 \pi \imu n x^*_3(z_2) \right) \exp\left( 2 \pi \imu (n k + m) x^1(z_2) \right)  \exp \left(- (x^2(z_2))^2 - 2kx^2(z_2) - k^2 \right) }	$}	\\
		&	=	&	\scalebox{0.85}{$	\sum_{k \in \mathbbm{Z}} \bk{\alpha^i_r,\exp \left( 2 \pi \imu n x^*_3(z_2) \right)} \exp\left( 2 \pi \imu (n k + m) x^1(z_2) \right)  \exp \left(- (x^2(z))^2 - 2kx^2(z_2) - k^2 \right)	$}	\\
		&	+	&	\scalebox{0.85}{$	\sum_{k \in \mathbbm{Z}} \exp \left( 2 \pi \imu n x^*_3(z_2) \right) \bk{\alpha^i_r,\exp\left( 2 \pi \imu (n k + m) x^1(z_2) \right)}  \exp \left(- (x^2(z_2))^2 - 2kx^2(z_2) - k^2 \right)		$}	\\
		&	+	&	\scalebox{0.85}{$	\sum_{k \in \mathbbm{Z}} \exp \left( 2 \pi \imu n x^*_3(z_2) \right) \exp\left( 2 \pi \imu (n k + m) x^1(z_2) \right)  \bk{\alpha^i_r, \exp \left(- (x^2(z_2))^2 - 2kx^2(z_2) - k^2 \right)}	$}	\\
		&	=	&	\scalebox{0.85}{$	\sum_{k \in \mathbbm{Z}} \bk{\alpha^i_r,\exp \left( 2 \pi \imu n x^*_3(z_2) \right)} \exp\left( 2 \pi \imu (n k + m) x^1(z_2) \right)  \exp \left(- (x^2(z_2))^2 - 2kx^2(z) - k^2 \right)		$}	\\
		&	=	&	\scalebox{0.85}{$	2 \pi \imu n z_2^r \delta_{i,3} \sum_{k \in \mathbbm{Z}} e^{2 \pi \imu n x^*_3(z_2)}  \exp\left( 2 \pi \imu (n k + m) x^1(z_2) \right)  \exp \left(- (x^2(z_2))^2 - 2kx^2(z_2) - k^2 \right)		$}	\\	
		&	=	&	2 \pi \imu n z_2^r \delta_{i,3} Y(F_m, z_2),
	\end{eqnarray*}

	so finally this leads to
	\begin{eqnarray*}
		\bk{\alpha^i(z_1), Y(F_m, z_2)}		&	=	&	\sum_{r \in \mathbbm{Z}} \bk{\alpha^i_r, Y(F_m, z_2)} z_1^{-1-r}	\\	
		&	=	&	2 \pi \imu n \delta_{i,3} Y(F_m, z_2) \sum_{r \in \mathbbm{Z}}  z_2^r  z_1^{-1-r}	=	2 \pi \imu n \delta_{i,3} Y(F_m, z_2) \delta(z_1, z_2).
	\end{eqnarray*}
	So the fields $\alpha^i(z_1), Y(F_m, z_2)$ are a local pair. 
	
	From \ref{TrivialCommutationRelation} and $\ref{NoTrivialCommutationRelation}$ follows that $\beta_3(z_1)$ and $Y(F_m, z_2)$ are a local pair.
	
	Let's prove that $\beta_2(z_1)$ and $Y(F_m, z_2)$ are local, from the relations proven in section \ref{section2.01} we deduce 
	\begin{eqnarray*}
		\bk{\beta_{2,r}, x^*_3(z_2)}					&	=	&	W^1 z_2^r \LOG{z_2},	\\
		\bk{\beta_{2,r}, x^1(z_2)}						&	=	&	0,				\\
		\bk{\beta_{2,r}, x^2(z_2)}						&	=	&	z_2^r K,		\\
		\bk{\beta_{2,r}, \left(x^2(z_2) \right)^2}		&	=	&	z_2^r	x^2(z_2),	
	\end{eqnarray*}
	from this follows
	\begin{eqnarray*}
		\bk{\beta_{2,r}, e^{2 \pi \imu n x^*_3(z_2)}}		&	=	&	2 \pi \imu n W^1 e^{2 \pi \imu nx^*_3(z_2)} z_2^r \LOG{z_2},	\\
		\bk{\beta_{2,r}, e^{2 \pi \imu (n k+m) x^1(z_2)}}	&	=	&	0,				\\
		\bk{\beta_{2,r}, e^{- 2 k x^2(z_2)}}				&	=	&	-2 k z_2^r e^{- 2 k x^2(z_2)},		\\
		\bk{\beta_{2,r}, e^{- \left( x^2(z_2) \right)^2}}	&	=	&	-2 z_2^r x^2(z_2) e^{- \left( x^2(z_2) \right)^2},	
	\end{eqnarray*}
	which finally leads to 
	\begin{eqnarray*}
		\bk{\beta_2 (z_1), e^{2 \pi \imu n x^*_3(z_2)}}		&	=	&	2 \pi \imu n W^1 e^{2 \pi \imu n x^*_3(z_2)} \LOG{z_2}	\delta(z_1, z_2),	\\
		\bk{\beta_2 (z_1), e^{2 \pi \imu (n k+m) x^1(z_2)}}	&	=	&	0,				\\
		\bk{\beta_2 (z_1), e^{- 2 k x^2(z_2)}}				&	=	&	-2 k e^{- 2 k x^2(z_2)}	\delta(z_1, z_2),		\\
		\bk{\beta_2 (z_1), e^{- \left( x^2(z_2) \right)^2}}	&	=	&	-2 x^2(z_2) e^{- \left( x^2(z_2) \right)^2}	\delta(z_1, z_2).	
	\end{eqnarray*}
	
	Now we can compute the brackets
	\begin{eqnarray*}
		\bk{\beta_2(z_1), Y(F_m, z_2)}		&	=	&	\scalebox{0.85}{$	\bk{\beta_2(z_1), \sum_{k \in \mathbbm{Z}} \exp \left( 2 \pi \imu n x^*_3(z_2) \right) \exp\left( 2 \pi \imu (n k + m) x^1(z_2) \right)  \exp \left(- (x^2(z_2))^2 - 2kx^2(z_2) - k^2 \right) }	$}	\\
		&	=	&	\scalebox{0.85}{$	\sum_{k \in \mathbbm{Z}} \bk{\beta_2(z_1),\exp \left( 2 \pi \imu n x^*_3(z_2) \right)} \exp\left( 2 \pi \imu (n k + m) x^1(z_2) \right)  \exp \left(- (x^2(z))^2 - 2kx^2(z_2) - k^2 \right)	$}	\\
		&	+	&	\scalebox{0.81}{$	\sum_{k \in \mathbbm{Z}} \exp \left( 2 \pi \imu n x^*_3(z_2) \right) \bk{\beta_2(z_1),\exp\left( 2 \pi \imu (n k + m) x^1(z_2) \right)}  \exp \left(- (x^2(z_2))^2 - 2kx^2(z_2) - k^2 \right)		$}	\\
		&	+	&	\scalebox{0.81}{$	\sum_{k \in \mathbbm{Z}} \exp \left( 2 \pi \imu n x^*_3(z_2) \right) \exp\left( 2 \pi \imu (n k + m) x^1(z_2) \right)  \bk{\beta_2(z_1), \exp \left(- (x^2(z_2))^2 \right)} \exp \left(- 2kx^2(z_2) - k^2 \right) 	$}	\\
		&	+	&	\scalebox{0.83}{$	\sum_{k \in \mathbbm{Z}} \exp \left( 2 \pi \imu n x^*_3(z_2) \right) \exp\left( 2 \pi \imu (n k + m) x^1(z_2) \right)   \exp \left(- (x^2(z_2))^2 \right) \bk{\beta_2(z_1), \exp \left(- 2kx^2(z_2)  \right)} \exp(- k^2) 	$}	\\
		&	=	&	\scalebox{0.80}{$ 2 \pi \imu n W^1 \LOG{z_2} \delta(z_1,z_2)	\sum_{k \in \mathbbm{Z}} \exp \left( 2 \pi \imu n x^*_3(z_2) \right) \exp\left( 2 \pi \imu (n k + m) x^1(z_2) \right)  \exp \left(- (x^2(z))^2 - 2kx^2(z_2) - k^2 \right)	$}	\\
		&	+	&	\scalebox{0.81}{$ -2  \delta(z_1,z_2) x^2(z_2)	\sum_{k \in \mathbbm{Z}} \exp \left( 2 \pi \imu n x^*_3(z_2) \right) \exp\left( 2 \pi \imu (n k + m) x^1(z_2) \right)  \exp \left(- (x^2(z))^2 - 2kx^2(z_2) - k^2 \right) 	$}	\\
		&	+	&	\scalebox{0.81}{$ -2 \delta(z_1,z_2)	\sum_{k \in \mathbbm{Z}} k \exp \left( 2 \pi \imu n x^*_3(z_2) \right) \exp\left( 2 \pi \imu (n k + m) x^1(z_2) \right)   \exp \left(- (x^2(z))^2 - 2kx^2(z_2) - k^2 \right) 	$}	\\	
		&	=	&	2 \pi \imu n W^1 Y(F_m, z_2) \LOG{z_2} \delta(z_1,z_2)	-	2 Y\left(\Theta_m(xe^{-x^2}), z_2 \right)	\delta(z_1,z_2),
	\end{eqnarray*}
	
	therefore $(z_1 - z_2)\bk{\beta_2(z_1), Y(F_m, z_2)} = 0$, so they are local.

	Let's prove that $\beta_1(z_1)$ and $Y(F_m, z_2)$ are local, once again from the relations proven in section \ref{section2.01} we get 
	\[
	\bk{\beta_{1,r}, x^*_3(z_2)}	=	z_2^r (x^2(z_2)	-	W^2 \LOG{z_2})	=	z_2^r	\tilde{x}^2(z_2),
	\]
	here we use the notation $\tilde{x}^i (z) = x^i(z) - W^i \LOG{z}	=	\sum_{n} x^i_n z^{-n}$, it holds
	\begin{eqnarray*}
		\bk{\beta_{1,r}, x^1(z_2)}						&	=	&	z_2^r K,		\\
		\bk{\beta_{1,r}, x^2(z_2)}						&	=	&	0,				\\
		\bk{\beta_{1,r}, \left(x^2(z_2) \right)^2}		&	=	&	0,	
	\end{eqnarray*}
	
	which means
	\begin{eqnarray*}
		\bk{\beta_{1,r}, e^{2 \pi \imu n x^*_3(z_2)}}		&	=	&	2 \pi \imu n z_2^r \tilde{x}^2(z_2) e^{2 \pi \imu n x^*_3(z_2)},	\\
		\bk{\beta_{1,r}, e^{2 \pi \imu (n k+m) x^1(z_2)}}	&	=	&	2 \pi \imu (nk+m) z_2^r e^{2 \pi \imu (nk+m) x^1(z_2)},				\\
		\bk{\beta_{1,r}, e^{- 2 k x^2(z_2)}}				&	=	&	0,		\\
		\bk{\beta_{1,r}, e^{- \left( x^2(z_2) \right)^2}}	&	=	&	0,
	\end{eqnarray*}
	
	and this translates into
	\begin{eqnarray*}
		\bk{\beta_1 (z_1), e^{2 \pi \imu n x^*_3(z_2)}}		&	=	&	2 \pi \imu n  \tilde{x}^2(z_2) \delta(z_1, z_2),	\\
		\bk{\beta_1 (z_1), e^{2 \pi \imu (n k+m) x^1(z_2)}}	&	=	&	2 \pi \imu (n k + m) e^{2 \pi \imu (n k + m) x^1(z_2)} \delta(z_1,z_2),				\\
		\bk{\beta_1 (z_1), e^{- 2 k x^2(z_2)}}				&	=	&	0,		\\
		\bk{\beta_1 (z_1), e^{- \left( x^2(z_2) \right)^2}}	&	=	&	0.	
	\end{eqnarray*}

	Now the commutator of the fields is computed
	\begin{eqnarray*}
		\bk{\beta_1(z_1), Y(F_m, z_2)}		&	=	&	\scalebox{0.85}{$	\bk{\beta_1(z_1), \sum_{k \in \mathbbm{Z}} \exp \left( 2 \pi \imu n x^*_3(z_2) \right) \exp\left( 2 \pi \imu (n k + m) x^1(z_2) \right)  \exp 	\left(- (x^2(z_2))^2 - 2kx^2(z_2) - k^2 \right) }	$}	\\
		&	=	&	\scalebox{0.85}{$	\sum_{k \in \mathbbm{Z}} \bk{\beta_1(z_1),\exp \left( 2 \pi \imu n x^*_3(z_2) \right)} \exp\left( 2 \pi \imu (n k + m) x^1(z_2) \right)  \exp \left(- (x^2(z))^2 - 2kx^2(z_2) - k^2 \right)	$}	\\
		&	+	&	\scalebox{0.81}{$	\sum_{k \in \mathbbm{Z}} \exp \left( 2 \pi \imu n x^*_3(z_2) \right) \bk{\beta_1(z_1),\exp\left( 2 \pi \imu (n k + m) x^1(z_2) \right)}  \exp \left(- (x^2(z_2))^2 - 2kx^2(z_2) - k^2 \right)		$}	\\
		&	+	&	\scalebox{0.81}{$	\sum_{k \in \mathbbm{Z}} \exp \left( 2 \pi \imu n x^*_3(z_2) \right) \exp\left( 2 \pi \imu (n k + m) x^1(z_2) \right)  \bk{\beta_1(z_1), \exp \left(- (x^2(z_2))^2 \right)} \exp \left(- 2kx^2(z_2) - k^2 \right) 	$}	\\
		&	+	&	\scalebox{0.83}{$	\sum_{k \in \mathbbm{Z}} \exp \left( 2 \pi \imu n x^*_3(z_2) \right) \exp\left( 2 \pi \imu (n k + m) x^1(z_2) \right)   \exp \left(- (x^2(z_2))^2 \right) \bk{\beta_1(z_1), \exp \left(- 2kx^2(z_2)  \right)} \exp(- k^2) 	$}	\\
		&	=	&	\scalebox{0.80}{$ 2 \pi \imu \tilde{x}^2(z_2) \delta(z_1,z_2)	\sum_{k \in \mathbbm{Z}} \exp \left( 2 \pi \imu n x^*_3(z_2) \right) \exp\left( 2 \pi \imu (n k + m) x^1(z_2) \right)  \exp \left(- (x^2(z))^2 - 2kx^2(z_2) - k^2 \right)	$}	\\
		&	+	&	\scalebox{0.81}{$ 2 \pi \imu (nk+m)  \delta(z_1,z_2)	\sum_{k \in \mathbbm{Z}} \exp \left( 2 \pi \imu n x^*_3(z_2) \right) \exp\left( 2 \pi \imu (n k + m) x^1(z_2) \right)  \exp \left(- (x^2(z))^2 - 2kx^2(z_2) - k^2 \right) 	$}	\\
		&	=	&	2 \pi \imu n \tilde{x}^2(z_2) Y(F_m, z_2) \delta(z_1,z_2)	+	2 \pi \imu m Y(F_m, z_2)	\delta(z_1,z_2)	-	n x^2(z_2) Y(F_m, z_2) \delta(z_1, z_2) \\
		&	+	&	\scalebox{0.8}{$ n2 \pi \imu (x^2(z_2) + k)  \delta(z_1,z_2)	\sum_{k \in \mathbbm{Z}} \exp \left( 2 \pi \imu n x^*_3(z_2) \right) \exp\left( 2 \pi \imu (n k + m) x^1(z_2) \right)  \exp \left(- (x^2(z))^2 - 2kx^2(z_2) - k^2 \right) 	$}\\
		&	=	&	\scalebox{0.85}{$ - 2 \pi \imu n W^2 Y(F_m, z_2) \LOG{z_2} \delta(z_1,z_2)	+	2 \pi \imu m Y(F_m, z_2) \delta(z_1,z_2) +	 2 \pi \imu nY( \Theta_m(x e^{-x^2}), z_2) \delta(z_1, z_2).$}
	\end{eqnarray*}
	
	The locality condition for the remaining fields follows from Dong's Lemma.\\
	

	Now it is only left to prove the translation invariance of the fields, we will proceed similarly to the proof of Theorem \ref{Hf0IsVertexAlgebra}, let us define the translation endomorphism $T$ in $V_N$. We already know how to define $T(x^i)$ and proceeding exactly as in the proof of \ref{Hf0IsVertexAlgebra} we get that the fields $e^{2 \pi \imu \rho_i x^i(z)}$ satisfy the translation invariance condition.
	
	For $T(x^*_3)$ the situation is similar but slightly more complicated, once again a vector such that $Y( T(x^*_3), z) = \partial_z Y(x^*_3,z)$ is needed, but unfortunately the equation \ref{DiffEquationBeta} is a little bit more complicated. We start noticing that after the change of coordinates we made the equation \ref{DiffEquationBeta}  was transformed into
	\begin{equation*} \label{DiffEquationBetaMorphed}
	D_z x^*_i(z)	= \beta_i (z) - \Hijk x^k(z) D_z x^j(z),
	\end{equation*}
	so taking $i=3$, acting on the vacuum vector and evaluating $z=0$ it becomes clear that $T(x^*_3)$ should be defined as 
	\[
	T(x^*_3)	=	\beta_{3,-1} \vacuum - \alpha^1_{-1} x^2_0 \vacuum,
	\]
	and force the commutation relation 
	\[
	\bk{T, x^*_{3,0}}	=	\beta_{3,-1}	-	\frac{\epsilon_{3jk}}{2} \sum_m m x^j_{-1 -m} x^k_m.
	\]
	
	Now it is easy to define $T$ on any function as
	\[
	T(f)	=	\partial_{x^1} f T(x^1)	+	\partial_{x^2} f T(x^2)	+	\partial_{x^*_3} f T(x^*_3).
	\]
	To make computations easier here we will actually use the fact that $W^1$, $W^2$ and $P^3$ act by zero so we have no logarithms in the fields.
	
	Consider the field $\tilde{x}^*_3 (z)	=	\sum_n x^*_3 z^{-n}$, it is convenient to prove that translation invariance holds for the field $\tilde{x}^*_3 (z)$, for $n \neq 0$ we have
	\begin{eqnarray*}
		\bk{T, x^*_{3.n}}	&	=	&	\bk{T, \frac{-\beta_{3,n}}{n} +  \frac{\epsilon_{3jk}}{2n} \sum_m m x^j_{n-m} x^k_{m} }		\\
		&	=	&	-\frac{1}{n}\bk{T, \beta_{3,n}}		+	\frac{\epsilon_{3jk}}{2n}	\sum_m m \bk{T, x^j_{n-m} x^k_{m} }		\\
		&	=	&	\beta_{3,n-1}	+	\frac{\epsilon_{3jk}}{2n}	\sum_m m \bk{T, x^j_{n-m}} x^k_{m} 	+	\frac{\epsilon_{3jk}}{2n}	\sum_m m x^j_{n-m} \bk{T, x^k_{m} }	\\
		&	=	&	\beta_{3,n-1}	+	\frac{\epsilon_{3jk}}{2n}	\sum_m m \alpha^j_{n-m-1} x^k_{m} 	+	\frac{\epsilon_{3jk}}{2n}	\sum_m m x^j_{n-m}  \alpha^k_{m-1}	\\
		&	=	&	\beta_{3,n-1}	-	\frac{\epsilon_{3jk}}{2n}	\sum_m m(n-m-1) x^j_{n-m-1} x^k_{m} 	-	\frac{\epsilon_{3jk}}{2n}	\sum_m m(m-1) x^j_{n-m}  x^k_{m-1}	\\
		&	=	&	\beta_{3,n-1}	-	\frac{\epsilon_{3jk}}{2n}	\sum_m m(n-m-1) x^j_{n-m-1} x^k_{m} 	-	\frac{\epsilon_{3jk}}{2n}	\sum_m (m+1)m x^j_{n-m-1}  x^k_{m}	\\
		&	=	&	\beta_{3,n-1}	-	\frac{\epsilon_{3jk}}{2}	\sum_m m x^j_{n-m-1} x^k_{m},
	\end{eqnarray*}
	then $\bk{T, \tilde{x}^*_3 (z)}$ expands as
	\begin{eqnarray*}
		\bk{T, \tilde{x}^*_{3}(z)}	&	=	&	\sum_n \bk{T, x^*_{3,n}} z^{-n}		\\
		&	=	&	\sum_n \beta_{3,n-1} z^{-n}	-	\frac{\epsilon_{3jk}}{2} \sum_n \sum_m m x^j_{n-m-1} x^k_{m} z^{-n}		\\		
		&	=	&	\sum_n \beta_{3,n} z^{-n-1}	-	\frac{\epsilon_{3jk}}{2} \sum_n \sum_m m x^j_{n-m} x^k_{m} z^{-n-1}		\\		
		&	=	&	\beta_3(z)	-	\frac{\epsilon_{3jk}}{2} \sum_n \sum_m  x^j_{n-m} z^{-n+m} m x^k_{m} z^{-m-1}		\\
		&	=	&	\beta_3(z)	+	\frac{\epsilon_{3jk}}{2} \left(\sum_n x^j_{n} z^{-n} \right)  \left( \sum_n -n x^k_{n} z^{-n-1}	\right)	\\
		&	=	&	\beta_3(z)	+	\frac{\epsilon_{3jk}}{2} x^j(z)  \partial_z  x^k(z)	\\
		&	=	&	\partial_z \tilde{x}^*_3(z),
	\end{eqnarray*}
	here the last equality hold because \ref{DiffEquationBeta} and $\tilde{x}^*_3(z)$ coincides with the $x^*_3(z)$ as defined in section \ref{section2.01} when setting the formal variable $\LOGz = 0$, i.e., when deleting all terms with $P_3$, $W^1$ and $W^2$. Now the field $x^*_3(z)$ after the change of coordinates (without the logarithmic terms) can be written as:
	\[
	x^*_3(z)	=	\tilde{x}^*_3(z)	+	\frac{1}{2} x^1(z)x^2(z),
	\]
	but now it becomes easy to prove translation invariance for $x^*_3(z)$ as we already know it holds for $x^1(z)$ and $x^2(z)$
	\begin{eqnarray*}
		\bk{T, x^*_{3}(z)}	&	=	&	\bk{T, \tilde{x}^*_3(z)}	+	\frac{1}{2} \bk{T, x^1(z)x^2(z)}		\\
		&	=	&	\partial_z \tilde{x}^*_3(z)	+	\frac{1}{2} \bk{T, x^1(z)}x^2(z)	+	\frac{1}{2} x^1(z) \bk{T, x^2(z)}		\\
		&	=	&	\partial_z \tilde{x}^*_3(z)	+	\frac{1}{2} \partial_z x^1(z)x^2(z)	+	\frac{1}{2} x^1(z) \partial_z x^2(z)	\\
		&	=	&	\partial_z \left(	\tilde{x}^*_3(z)	+	\frac{1}{2} x^1(z)x^2(z)	\right)	\\
		&	=	&	\partial_z x^*_3(z).
	\end{eqnarray*}
	
	Finally we have the tools for proving the translation invariance condition for the fields $Y(F_m, z)$
	\begin{eqnarray*}
		\bk{T, Y(F_m, z)}	&	=	&	\bk{T, \sum_{k \in \mathbbm{Z}} \exp \left( 2 \pi \imu n x^*_3(z) \right) \exp\left( 2 \pi \imu (n k + m) x^1(z) \right)  \exp \left(- (x^2(z))^2 - 2kx^2(z) - k^2 \right) }		\\
		&	=	&	\sum_{k \in \mathbbm{Z}} \bk{T, \exp \left( 2 \pi \imu n x^*_3(z) \right)} \exp\left( 2 \pi \imu (n k + m) x^1(z) \right)  \exp \left(- (x^2(z))^2 - 2kx^2(z) - k^2 \right)	\\	
		&	+	&	\sum_{k \in \mathbbm{Z}} \exp \left( 2 \pi \imu n x^*_3(z) \right) \bk{T, \exp\left( 2 \pi \imu (n k + m) x^1(z) \right)}  \exp \left(- (x^2(z))^2 - 2kx^2(z) - k^2 \right)	\\
		&	+	&	\sum_{k \in \mathbbm{Z}} \exp \left( 2 \pi \imu n x^*_3(z) \right) \exp\left( 2 \pi \imu (n k + m) x^1(z) \right)  \bk{T, \exp \left(- (x^2(z))^2 - 2kx^2(z) - k^2 \right)}	\\
		&	=	&	\scalebox{0.85}{$	\sum_{k \in \mathbbm{Z}} 2 \pi \imu n\bk{T, x^*_3(z)} \exp \left( 2 \pi \imu n x^*_3(z) \right) \exp\left( 2 \pi \imu (n k + m) x^1(z) \right)  \exp \left(- (x^2(z))^2 - 2kx^2(z) - k^2 \right)	 $}		\\	
		&	+	&	\scalebox{0.85}{$	\sum_{k \in \mathbbm{Z}} 2 \pi \imu (n k + m)\exp \left( 2 \pi \imu n x^*_3(z) \right) \bk{T, x^1(z)} \exp\left( 2 \pi \imu (n k + m) x^1(z) \right)  \exp \left(- (x^2(z))^2 - 2kx^2(z) - k^2 \right)	$}	\\
		&	+	&	\scalebox{0.85}{$	\sum_{k \in \mathbbm{Z}} -2(x^2(z)+k) \exp \left( 2 \pi \imu n x^*_3(z) \right) \exp\left( 2 \pi \imu (n k + m) x^1(z) \right)  \bk{T, x^2(z)}\exp \left(- (x^2(z))^2 - 2kx^2(z) - k^2 \right)		$}		\\
		&	=	&	\scalebox{0.85}{$	\sum_{k \in \mathbbm{Z}} 2 \pi \imu n \partial_z \left( x^*_3(z) \right) \exp \left( 2 \pi \imu n x^*_3(z) \right) \exp\left( 2 \pi \imu (n k + m) x^1(z) \right)  \exp \left(- (x^2(z))^2 - 2kx^2(z) - k^2 \right)	 $}		\\	
		&	+	&	\scalebox{0.85}{$	\sum_{k \in \mathbbm{Z}} 2 \pi \imu (n k + m)\exp \left( 2 \pi \imu n x^*_3(z) \right) \partial_z \left( x^1(z) \right) \exp\left( 2 \pi \imu (n k + m) x^1(z) \right)  \exp \left(- (x^2(z))^2 - 2kx^2(z) - k^2 \right)	$}	\\
		&	+	&	\scalebox{0.85}{$	\sum_{k \in \mathbbm{Z}} -2(x^2(z)+k) \exp \left( 2 \pi \imu n x^*_3(z) \right) \exp\left( 2 \pi \imu (n k + m) x^1(z) \right)  \partial_z \left( x^2(z)\right)  \exp \left(- (x^2(z))^2 - 2kx^2(z) - k^2 \right)		$}		\\
		&	=	&	\scalebox{0.85}{$	\sum_{k \in \mathbbm{Z}} \partial_z \left(  \exp \left( 2 \pi \imu n x^*_3(z) \right) \right) \exp\left( 2 \pi \imu (n k + m) x^1(z) \right)  \exp \left(- (x^2(z))^2 - 2kx^2(z) - k^2 \right)	 $}		\\	
		&	+	&	\scalebox{0.85}{$	\sum_{k \in \mathbbm{Z}} \exp \left( 2 \pi \imu n x^*_3(z) \right) \partial_z \left(  \exp\left( 2 \pi \imu (n k + m) x^1(z) \right) \right)  \exp \left(- (x^2(z))^2 - 2kx^2(z) - k^2 \right)	$}	\\
		&	+	&	\scalebox{0.85}{$	\sum_{k \in \mathbbm{Z}} \exp \left( 2 \pi \imu n x^*_3(z) \right) \exp\left( 2 \pi \imu (n k + m) x^1(z) \right)  \partial_z \left(   \exp \left(- (x^2(z))^2 - 2kx^2(z) - k^2 \right)\right)		$}		\\
		&	=	&	\partial_z \left( \sum_{k \in \mathbbm{Z}} \exp \left( 2 \pi \imu n x^*_3(z) \right) \exp\left( 2 \pi \imu (n k + m) x^1(z) \right)  \exp \left(- (x^2(z))^2 - 2kx^2(z) - k^2 \right)  \right)		\\
		&	=	&	\partial_z Y(F_m, z).
	\end{eqnarray*}

	
\end{proof}

We have the embeddings
\[
\Complexn	\subset	C^{\infty} (N)	\subset		C^{\infty}(G / \Gamma),
\]
which leads to
\[
V^1(\Gg) \vacuum	\subset V_N \subset		\Hc,
\]
so it is expected that $\Hc$ is a logarithmic module over $V_N$.

\begin{theorem}	\label{HF1Module}
	The space $\Hc$ has the structure of logarithmic $V_N$-module.
\end{theorem}

\begin{proof}
	We must define a logarithmic module for each vector of $V_N$, set
	\begin{eqnarray*}
		Y( e^{2\pi \imu \rho_i x^i}, z)		&	=	&	e^{2\pi \imu \rho_i x^i(z)},	\\
		Y( \alpha^i_{-1}\vacuum, z)			&	=	&	\alpha^i(z),	\\ 
		Y( \beta_{i,-1}\vacuum, z)			&	=	&	\beta_i(z),		\\
		Y(F_m, z)							&	=	&	\sum_{k \in \mathbbm{Z}} \exp \left( 2 \pi \imu n x^*_3(z) \right) \exp\left( 2 \pi \imu (n k + m) x^1(z) \right)  \exp \left(- (x^2(z))^2 - 2kx^2(z) - k^2 \right),
	\end{eqnarray*}
	Now we extend $Y$ to any function $f \in \bigoplus_{n \neq 0} S(\mathbbm{R}) \otimes \Complexn^{\Abs{n}}$ exactly as we did in the previous theorem \ref{Hf1IsVertexAlgebra}, i.e., through the formula \ref{EquationForDefiningFieldsofFunctions}, in particular for every $a \in \Gg$ the logarithmic field $Y(a F_m,z)$ is defined by the formula $Y(a F_m, z) = \bk{a_0, F_m}$ and finally we extend $Y$ to the rest of the vectors via the normally ordered product.
	
	Notice that during all the analysis made in the proof of \ref{Hf1IsVertexAlgebra} to show that $Y(f)$ was actually a field was never used the fact that the logarithmic terms acted by zero, so what we actually prove back there was that the $Y(f)$ for any function was actually a logarithmic field. Similarly we proceeded, on propose, when proving that the fields were pairwise local, so what we actually prove was that those are pairwise local logarithmic fields.

	Let's prove that the function $Y: V_N \rightarrow \Lfield{\Hc}$ preserves the $n$-products, because of the way we defined the fields and equation \ref{DerivativeNProduct} it is clear that $Y$ preserves all negative $n$-products. For positive $n$-products  involving only the fields $e^{2\pi \imu \rho_i x^i(z)}$, $\alpha^i(z)$, $\beta_i(z)$ the $n$-product condition holds, the analysis is complete analogous to the one made in the proof of theorem \ref{HF0Module}. It is also clear that for $n \geq 0$
	\[
	Y({F_{m_1}}_{(n)} F_{m_2}, z)	=	0	=	Y(F_{m_1}, z)_{(n)} Y(F_{m_2}, z)
	\]
	because $Y(F_{m_1}, z)$ and $Y(F_{m_2}, z)$ commute.
	
	For $\alpha^i(z)$ and $Y(F_m,z)$ we have
	\begin{eqnarray*}
		\alpha^i(z)_{(0)} Y(F_m,z)	&	=	&	\left.(z_1-z_2)\bk{\alpha^i(z_1)_{-}, Y(F_m,z_2)} \right\vert_{z_1=z_2 = z}	\\	
		&	=	&	2 \pi \imu n	\delta_{i,3} Y(F_m, z)	\\
		&	=	&	Y(\alpha^i_{(o)} F_m, z).
	\end{eqnarray*}
	
	For the fields $\beta_3(z)$ and $Y(F_m,z)$ there is nothing to prove since they commute.
	
	For the fields $\beta_2(z)$ and $Y(F_m,z)$ holds
	\begin{eqnarray*}
		\beta_2(z)_{(0)} Y(F_m,z)	&	=	&	\left.(z_1-z_2)\bk{\beta_2(z_1)_{-}, Y(F_m,z_2)} \right\vert_{z_1=z_2 = z}	\\	
		&	=	&	2 \pi \imu n W^1 Y(F_m, z_2) \LOG{z_2}	-	2 Y\left(\Theta_m(xe^{-x^2}), z_2 \right)	\\
		&	=	&	\bk{\beta_{2, 0}, Y(F_m, z) }	\\
		&	=	&	Y({\beta_2}_{(0)} F_m, z).
	\end{eqnarray*}
	
	Finally for $\beta_1(z)$ and $Y(F_m,z)$ holds
	\begin{eqnarray*}
		\beta_1(z)_{(0)} Y(F_m,z)	&	=	&	\left.(z_1-z_2)\bk{\beta_1(z_1)_{-}, Y(F_m,z_2)} \right\vert_{z_1=z_2 = z}	\\	
		&	=	&	\scalebox{0.85}{$ - 2 \pi \imu n W^2 Y(F_m, z_2) \LOG{z_2} 	+	2 \pi \imu m Y(F_m, z_2)  +	 2 \pi \imu nY( \Theta_m(x e^{-x^2}), z_2) .$}	\\
		&	=	&	\bk{\beta_{1, 0}, Y(F_m, z) }	\\
		&	=	&	Y({\beta_1}_{(0)} F_m, z).
	\end{eqnarray*}
	
	Therefore $\Hc$	is a logarithmic $V_N$-module.
	
\end{proof}




\begin{thebibliography}{10}

\bibitem{malikov1999chiral}
Fyodor Malikov, Vadim Schechtman, and Arkady Vaintrob.
\newblock Chiral de rham complex.
\newblock {\em Communications in mathematical physics}, 204(2):439--473, 1999.

\bibitem{beilinson2004chiral}
Alexander Beilinson and Vladimir Drinfeld.
\newblock {\em Chiral algebras}, volume~51.
\newblock American Mathematical Soc., 2004.

\bibitem{song2016vector}
Bailin Song.
\newblock Vector bundles induced from jet schemes.
\newblock {\em arXiv preprint arXiv:1609.03688}, 2016.

\bibitem{borisov2000elliptic}
Lev~A Borisov and Anatoly Libgober.
\newblock Elliptic genera of toric varieties and applications to mirror
  symmetry.
\newblock {\em Inventiones mathematicae}, 140(2):453--485, 2000.

\bibitem{aldi2012dilogarithms}
Marco Aldi and Reimundo Heluani.
\newblock Dilogarithms, ope, and twisted t-duality.
\newblock {\em International Mathematics Research Notices}, 2014(6):1528--1575,
  2012.

\bibitem{hull2009double}
Chris Hull and Barton Zwiebach.
\newblock Double field theory.
\newblock {\em Journal of High Energy Physics}, 2009(09):099, 2009.

\bibitem{bouwknegt2004t}
Peter Bouwknegt, Jarah Evslin, and Varghese Mathai.
\newblock T-duality: topology change from h-flux.
\newblock {\em Communications in mathematical physics}, 249(2):383--415, 2004.

\bibitem{bakalov2016twisted}
Bojko Bakalov.
\newblock Twisted logarithmic modules of vertex algebras.
\newblock {\em Communications in Mathematical Physics}, 345(1):355--383, 2016.

\bibitem{kac1998vertex}
Victor~G Kac.
\newblock {\em Vertex algebras for beginners}.
\newblock Number~10. American Mathematical Soc., 1998.

\bibitem{de2006finite}
Alberto De~Sole and Victor~G Kac.
\newblock Finite vs affine w-algebras.
\newblock {\em Japanese Journal of Mathematics}, 1(1):137--261, 2006.

\bibitem{carter2005lie}
Roger~William Carter.
\newblock {\em Lie algebras of finite and affine type}, volume~96.
\newblock Cambridge University Press, 2005.

\bibitem{auslander2006abelian}
Louis Auslander and Richard Tolimieri.
\newblock {\em Abelian Harmonic Analysis, Theta Functions and Functional
  Algebras on a Nilmanifold}, volume 436.
\newblock Springer, 2006.

\bibitem{folland2016course}
Gerald~B Folland.
\newblock {\em A course in abstract harmonic analysis}.
\newblock Chapman and Hall/CRC, 2016.

\end{thebibliography}
\end{document}